\documentclass[12pt]{amsart}

\usepackage{psfrag}
\usepackage{color}

\usepackage{graphicx,graphics}
\usepackage{fullpage,amssymb,amsfonts,amsmath,amstext,amsthm,amscd,enumerate,verbatim,tikz,booktabs,subcaption,orcidlink}
\usepackage[T1]{fontenc}
\usetikzlibrary{matrix,arrows}

\begin{document}
\date{\today}

\newtheorem{theorem}{Theorem}[section]
\newtheorem{result}[theorem]{Result}
\newtheorem{fact}[theorem]{Fact}
\newtheorem{conjecture}[theorem]{Conjecture}
\newtheorem{definition}[theorem]{Definition}
\newtheorem{lemma}[theorem]{Lemma}
\newtheorem{proposition}[theorem]{Proposition}
\newtheorem{remark}[theorem]{Remark}
\newtheorem{corollary}[theorem]{Corollary}
\newtheorem{facts}[theorem]{Facts}
\newtheorem{question}[theorem]{Question}
\newtheorem{props}[theorem]{Properties}

\newtheorem{ex}[theorem]{Example}

\newcommand{\notes} {\noindent \textbf{Notes.  }}
\newcommand{\note} {\noindent \textbf{Note.  }}
\newcommand{\defn} {\noindent \textbf{Definition.  }}
\newcommand{\defns} {\noindent \textbf{Definitions.  }}
\newcommand{\x}{{\bf x}}
\newcommand{\z}{{\bf z}}
\newcommand{\B}{{\bf b}}
\newcommand{\V}{{\bf v}}
\newcommand{\T}{\mathcal{T}}
\newcommand{\Z}{\mathbb{Z}}
\newcommand{\Hp}{\mathbb{H}}
\newcommand{\D}{\mathbb{D}}
\newcommand{\R}{\mathbb{R}}
\newcommand{\N}{\mathbb{N}}
\renewcommand{\B}{\mathbb{B}}
\newcommand{\C}{\mathbb{C}}
\newcommand{\dt}{{\mathrm{det }\;}}
 \newcommand{\adj}{{\mathrm{adj}\;}}
 \newcommand{\0}{{\bf O}}
 \newcommand{\av}{\arrowvert}
 \newcommand{\zbar}{\overline{z}}
 \newcommand{\htt}{\widetilde{h}}
\newcommand{\ty}{\mathcal{T}}
\renewcommand\Re{\operatorname{Re}}
\renewcommand\Im{\operatorname{Im}}
\newcommand{\diam}{\operatorname{diam}}
\newcommand{\sign}{\operatorname{sign}}
\newcommand{\ft}{\widetilde{f}}
\newcommand{\Rot}{\mathcal{R}}
\renewcommand{\S}{\mathbb{S}}
\newcommand{\dist}{\operatorname{dist}}

\newcommand{\ds}{\displaystyle}
\numberwithin{equation}{section}

\renewcommand{\theenumi}{(\roman{enumi})}
\renewcommand{\labelenumi}{\theenumi}

\title{Schwarz-Pick and Denjoy-Wolff for uniformly quasiregular maps}

\author{Alastair N. Fletcher\,\orcidlink{0000-0003-1942-6928}}
\email{afletcher@niu.edu}
\address{Department of Mathematical Sciences, Northern Illinois University, DeKalb, IL 60115-2888, USA}


\begin{abstract}
We prove a Denjoy-Wolff theorem for uniformly quasiregular self-maps of the unit ball in $\mathbb{R}^n$, $n\geq 2$: either all orbits are relatively compact, or the iterates converge locally uniformly to a single point of the sphere. No properness, surjectivity, or boundary extension is assumed. The key tool is a Schwarz-Pick lemma valid in all dimensions: every quasiregular map preserving a bounded measurable conformal structure is non-expanding for a modulus metric built from that structure, and uniformly quasiregular maps preserve such a structure by a theorem of Iwaniec and Martin. Since this metric is not geodesic, we pass to an associated chain metric, which we show is Gromov hyperbolic with boundary the sphere, and apply a theorem of Karlsson. As an application, a stable Fatou component of a uniformly quasiregular map of $\overline{\mathbb{R}^n}$ that is a uniform domain with uniformly perfect boundary is an attracting basin, a rotation domain, or a parabolic basin with a unique boundary fixed point. Non-injective examples of both escaping types are constructed in every dimension.
\end{abstract}

\maketitle

\section{Introduction}

The Denjoy-Wolff Theorem states that if $f:\D \to \D$ is holomorphic, then either $f$ is an elliptic automorphism, or there exists a unique $w\in \overline{\D}$ such that $f^m \to w$ locally uniformly on $\D$. Generalizations to many other settings abound, for example \cite{Aba88,Bea97,Hei41,Her63,KKR,Kar01,LLNW,Mac83}. The starting point of every proof is the Schwarz-Pick Lemma: holomorphic self-maps of $\D$ are non-expanding with respect to the hyperbolic metric. For quasiregular maps there are versions of the Schwarz-Pick Lemma, but all of those the author is aware of control the distortion of distances rather than give non-expansion, and for a metric independent of the map nothing better is possible: by \cite[Proposition 7.1]{FM20}, uniformly quasiregular self-maps of $\B^n$ can expand the hyperbolic metric by any prescribed factor. Finding a metric, adapted to the map, for which a uniformly quasiregular map is non-expanding is one of the main aims of this paper.

Of relevance to this paper, the author and Macclure showed in \cite[Theorem 1.8]{FM20} that if $f:\B^3 \to \B^3$ is a surjective, proper uniformly quasiregular map, then either the iterates form a semi-group of automorphisms of $\B^3$, or there exists a unique $w\in \overline{\B^3}$ such that $f^m \to w$ locally uniformly on $\B^3$. The idea behind this result is that the hypotheses imply that $f$ extends to a uniformly quasiregular map of $\S^2$ to $\S^2$. If the result did not hold, then there would be a continuum of accumulation points of $f$ in $\S^2$, all fixed by $f$. This boundary map is nothing other than a quasiconformal conjugate of a rational map, and a rational map that is not the identity certainly cannot fix a non-trivial continuum of points. The paper \cite{FM20}, slightly paraphrased for clarity, notes that {\it ``this is the first version of the Denjoy-Wolff Theorem which does not directly use non-expansiveness with respect to a metric.''} The aim of this paper is to show that this somewhat naive view can be corrected.

To that end, we introduce a new metric based on measurable conformal structures. We give full definitions below, but for now, if $n\geq 2$ and $f:\B^n\to \B^n$ is uniformly quasiregular, then Iwaniec and Martin \cite{IM96} showed that there is a measurable conformal structure $G$ such that $f^*G = G$. One interpretation of this is that $f$ is conformal with respect to the structure $G$. 

For our first main result, the Schwarz-Pick Lemma shows that every holomorphic map $f:\D\to \D$ is either a strict hyperbolic contraction, or a hyperbolic isometry and a M\"obius automorphism of $\D$. For quasiregular maps in $\B^n$, there are versions of the Schwarz-Pick Lemma which bound how far points can be pushed apart in the hyperbolic metric in terms of the maximal dilatation and the dimension, see \cite[Chapter 16]{HKV}. Here we give a closer analogue of the Schwarz-Pick Lemma for quasiregular maps which are conformal with respect to a measurable conformal structure.

\begin{theorem}[Schwarz-Pick Lemma]
\label{thm:sp}
Let $n\geq 2$, let $U$ be a proper subdomain of $\R^n$ with $\operatorname{cap} \partial U >0$, let $G$ be a bounded measurable conformal structure on $U$, and let $f:U\to U$ be a non-constant quasiregular map with $f^*G = G$ almost everywhere. Then the invariant modulus metric $d_{G}$ of Definition \ref{def:Gmetric} satisfies
\begin{equation}
\label{eq:sp1} 
d_G (f(x) , f(y) ) \leq d_G(x,y) 
\end{equation}
for all $x,y\in U$. Moreover,
\begin{enumerate}[(a)]
\item if $f$ is a homeomorphism of $U$ onto itself, then $f$ is a $d_G$-isometry;
\item if $f$ is a $d_G$-isometry, then $f$ is injective;
\item if $f$ is not injective, then there are points $x\neq y$ with strict inequality in \eqref{eq:sp1}.
\end{enumerate}
\end{theorem}

In particular, Theorem \ref{thm:sp} applies to every uniformly quasiregular self-map $f:U \to U$, where $G$ is the Iwaniec-Martin invariant structure. The metric $d_{G}$ has to be calibrated to $f$. As noted above, no metric independent of $f$ can work. We emphasize that the hypothesis of Theorem \ref{thm:sp} concerns a single map preserving a single structure: it applies, with the same metric $d_G$, to every member of a uniformly quasiregular semigroup, since Iwaniec and Martin \cite{IM96} construct a structure invariant under the whole semigroup, and it applies to every quasiregular map preserving $G$, uniformly quasiregular or not.

Parts~(a)-(c) in Theorem \ref{thm:sp} are the analogue of the classical dichotomy for holomorphic maps in $\D$ that are either conformal automorphisms and hence isometries, or not an automorphism and hence strict contractions, but with two differences. First, we do not know if an isometric injective uniformly quasiconformal map $\B^n \to \B^n$ must be surjective. For $G\equiv I$, this follows from the fact that orientation-preserving isometries of the modulus metric $\mu_{\B^n}$ are conformal maps (Betsakos and Pouliasis \cite{BP19} for $n=2$ and Zhang \cite{Zha21} for $n\geq 3$) and hence M\"obius. Second,~(c) says the contraction is strict somewhere, rather than strict everywhere. We do not know the full answer to whether we can conclude strict contraction in the non-injective case.

It is worth pointing out that when $G \neq I$, the isometry group of $d_G$ can be wilder than the M\"obius group. Mayer \cite{May93} studied an earlier construction of Tukia that gave uniformly quasiconformal self-maps of $\mathbb{H}^3$ that are topologically but not quasiconformally conjugate to a translation. If $h$ is such a map, then by~(a) in Theorem \ref{thm:sp}, $h$ is a surjective isometry of its invariant modulus metric, and it is not the restriction of a M\"obius map.

For our Denjoy-Wolff Theorem, we will need to take a chain metric variant $d$ of $d_G$, essentially because for $n\geq 3$ we do not know if $(\B^n , d_G)$ is Gromov hyperbolic, whereas we do for $d$. 

\begin{theorem}[Denjoy-Wolff Theorem]
\label{thm:dw}
Let $f:\B^n \to \B^n$ be uniformly quasiregular and let $d$ be the chain metric of Definition \ref{def:chain}. Then $d$ is a proper metric inducing the Euclidean topology, $(\B^n , d)$ is Gromov hyperbolic with Gromov boundary $\S^{n-1}$, $f$ is $d$-non-expanding, and exactly one of the following holds.
\begin{enumerate}[(a)]
\item Every orbit is relatively compact in $\B^n$. Then either $f^m \to x_0 \in \B^n$ locally uniformly, where $x_0$ is the unique fixed point, attracting or super-attracting, or $f$ is a uniformly quasiconformal homeomorphism of $\B^n$ onto itself, a $d$-isometry, and some subsequence $f^{k_j} \to \operatorname{id}$.
\item There is $w\in \S^{n-1}$ with $f^m \to w$ locally uniformly on $\B^n$. Moreover, $|f^m(x) - f^m(y)| \to 0$ for all $x,y$. For each $x\in \B^n$, the step sizes $s_m = d(f^m(x) , f^{m+1}(x))$ are non-increasing with limit $s\geq 0$, and there is a horofunction $h$ of $(\B^n , d)$ with $h(x) = 0$ such that the orbit of every $y$ lies in $\{ h \leq d(x,y) \}$.
\end{enumerate}
\end{theorem}

We call alternative~(a) the bounded case and~(b) the escaping case.
In the escaping case, we call it {\it positive-step} if $s>0$ and {\it zero-step} if $s=0$. Note that there are choices in deriving $d$ from $d_G$, but these choices do not affect the positive-step/zero-step classification, see the remark at the end of Section 6. We refer to, for example, \cite{Cow81,Pom79} for this classification in complex dynamics. Stated in the form in which Wolff's theorem is usually quoted, Theorem \ref{thm:dw} says the following.

\begin{corollary}
\label{cor:wolff}
Let $f:\B^n\to\B^n$ be uniformly quasiregular. If $f$ has no fixed point in $\B^n$ and is not a uniformly quasiconformal homeomorphism of $\B^n$ onto itself, then there is $w\in\S^{n-1}$ with $f^m\to w$ locally uniformly on $\B^n$.
\end{corollary}

For $n=2$, uniformly quasiregular self-maps of $\D$ are quasiconformally conjugate to holomorphic ones, and Theorems \ref{thm:sp} and \ref{thm:dw} reduce to the classical Schwarz-Pick Lemma and Denjoy-Wolff Theorem; see the remark at the end of Section 7.

This analogue of the Denjoy-Wolff Theorem has applications to stable Fatou components of uniformly quasiregular maps. Hinkkanen, Martin and Mayer \cite[Proposition 4.9]{HMM} showed that a stable Fatou component (see Section \ref{ss:uqr} for the terminology) is either an attracting basin, a parabolic basin, or a rotation domain, but note explicitly that ``we have not proved that a given $\Omega$ could be a parabolic basin on account of at most one $x_0\in\partial\Omega$'', that is, that the boundary fixed point which is a subsequential limit of the iterates might not be unique. Our next result addresses this point.

\begin{theorem}
\label{thm:stable}
Let $f:\S^n \to \S^n$ be a uniformly quasiregular map of degree at least $2$ and let $U$ be a stable component of $F(f)$ which is a uniform domain with uniformly perfect boundary (in the spherical metric). Then $U$ is either a (super-)attracting basin, a rotation domain, or a parabolic basin. In the last case, there is a unique $w \in \partial U$ which is a fixed point of $f$, with $f^m \to w$ locally uniformly on $U$.
\end{theorem}

The hypotheses here include every quasiball, every uniform domain with connected boundary, every uniform domain whose boundary has finitely many non-degenerate components, and every uniform domain $U$ with $\partial U = J(f)$. The proof is quantitative: the constants in it depend only on $n$, the uniformity constant of $U$, and the uniform perfectness constant and spherical diameter of $\partial U$.

Next, Blaschke products provide a rich family of holomorphic self-maps of $\D$ with various different behaviours of the Denjoy-Wolff point. It is natural to ask whether examples illustrating Theorem \ref{thm:dw} exist. For any $n\geq 2$, the uniformly quasiregular power maps constructed by Mayer \cite{May97} provide examples of uniformly quasiregular self-maps of $\B^n$ where $f^m \to 0$ locally uniformly in $\B^n$. For the escaping case, we state the following result. Here we say that a M\"obius map $g$ is a {\it conformal model} for $f$ at the Denjoy-Wolff point $w$ if there is an open set $V\subset \B^n$ with $w\in \overline{V}$, $f(\B^n) \subset V$ and $f = g$ on $V$, in the ball model or in any M\"obius equivalent model of $\B^n$; then $f^m = g^{m-1} \circ f$ for all $m\geq 1$, so the dynamics of $f$ is that of $g$ after one step.

\begin{theorem}
\label{thm:examples}
For every $n\geq 2$, there exist non-injective uniformly quasiregular self-maps of $\B^n$ in the escaping case, of positive-step type, whose conformal model at the Denjoy-Wolff point $w$ is a hyperbolic M\"obius map, or a parabolic one.
\end{theorem}

The hyperbolic case is based on a conformal trap method as introduced by Iwaniec and Martin \cite{IM96}, and the parabolic case is modelled after a construction of Hinkkanen, Martin and Mayer \cite[Theorem 5.2]{HMM}.

Finally, the passage from $d_G$ to the chain metric $d$ is a device of independent interest in the coarse geometry of hyperbolic-type metrics. The modulus metric $\mu_U$ of a domain, like several other conformally invariant metrics, is defined by a point-pair function and is not a length metric, so Gromov hyperbolicity cannot be transported to it from the quasihyperbolic metric by quasi-isometry; Lind\'en \cite{Lin07} proved that $(\B^n,\mu_{\B^n})$ is Gromov hyperbolic by a direct computation and asked whether the same holds for general domains, and H\"ast\"o \cite{Has06} showed that for the closely related metric $j_U$ the answer is negative. The chain metric replaces $\mu_U$ by a roughly geodesic metric at bounded distance from it at large scales, and a modest extension of a theorem of V\"ais\"al\"a \cite[Theorem 3.18]{Vai05} from intrinsic to roughly geodesic spaces, Theorem \ref{thm:gromovchain}, then applies. In particular, we obtain the following partial answer to Lind\'en's question.

\begin{theorem}
\label{thm:coarse}
Let $U\subset\S^n$ be a uniform domain with uniformly perfect boundary and let $d$ be a chain metric built from $\mu_U$ as in Definition \ref{def:chain}, with $r$ sufficiently large. Then $(U,d)$ is Gromov hyperbolic, quasi-isometric to $(U,\mu_U)$ and to $U$ with its quasihyperbolic metric, and its Gromov boundary is $\partial U$.
\end{theorem}

The paper is organized as follows. In Section 2 we cover preliminary material, in Section 3 we introduce the $G$-modulus and cover some of its basic properties. The invariant modulus metric $d_G$ is defined in Section 4. The Schwarz-Pick Lemma, Theorem \ref{thm:sp} is proved in Section 5 and the Denjoy-Wolff Theorem in Section 7 after the chain metric built from $d_G$ is constructed in Section 6. Section 8 contains the proofs of Theorems \ref{thm:stable} and \ref{thm:coarse}. Section 9 contains the examples for Theorem \ref{thm:examples}. Finally, Section 10 collects some open questions.

{\bf Acknowledgements and AI disclosure:} Claude Fable 5.1 was used for proofreading, grammatical correction, and as an exploratory tool for discussing possible approaches to selected results under the author's mathematical supervision and guidance. The author takes full responsibility for all mathematical arguments and for the accuracy and correctness of the final manuscript.

\section{Preliminaries}

\subsection{Modulus of path families}

The standard reference for material on the modulus of path families is V\"ais\"al\"a's book \cite{Vaisala}. We refer there for more on topics we can only briefly cover here. For a family $\Gamma$ of paths in $\R^n$, for $n\geq 2$, the conformal modulus is
\[ M(\Gamma) = \inf  \int \eta^n \: dm \]
over Borel densities $\eta \geq 0$ such that every locally rectifiable $\gamma \in \Gamma$ has $\eta$-length at least $1$, that is, $\int_{\gamma} \eta \: ds \geq 1$. If we need to consider path families on $\S^n$, we can do so by ignoring paths through $\infty$, as the family of all paths through a fixed point has modulus zero. 

We say that $\Gamma_2$ is minorized by $\Gamma_1$, written $\Gamma_1<\Gamma_2$ if every $\gamma \in \Gamma_2$ has a subpath in $\Gamma_1$. Then $M(\Gamma_1) \geq M(\Gamma_2)$. If $\Gamma_i$, for $i\in \N$, is a sequence of path families, then the countable subadditivity property says that $M( \bigcup_i \Gamma_i) \leq \sum_i M(\Gamma_i)$.

By $\Delta(E,F;U)$ we denote the path family joining $E$ to $F$ while staying in the domain $U$, apart from possibly the endpoints.

\subsection{Quasiregular maps}
\label{ss:qrmaps}

We refer to Rickman's monograph \cite{Rickman} for foundational material on quasiregular maps. The inner dilatation and outer dilatation of $f$ are denoted by $K_I(f)$ and $K_O(f)$ respectively, and the maximal dilatation is $K(f) = \max \{ K_I(f) , K_O(f) \}$. Poletsky's inequality \cite[Theorem II.8.1]{Rickman} states that for a non-constant quasiregular map $f:U\to \R^n$ and any path family $\Gamma$ in $U$, we have
\[ M(f(\Gamma)) \leq K_I(f) M(\Gamma),\]
where $f(\Gamma) = \{ f\circ \gamma : \gamma \in \Gamma \}$.

If $f:U\to\R^n$ is quasiregular, then it is discrete and open, and thus has the following path lifting property, see \cite[Section II.3]{Rickman}. If $\beta:[a,b) \to \R^n$ is a path and $x_0 \in f^{-1}(\beta(a))$, there is a maximal $f$-lift $\alpha : [a,c) \to U$ of $\beta$ starting at $x_0$, where $c\leq b$ and $f\circ \alpha = \beta$ on $[a,c)$. If $c<b$ then $\dist (\alpha(t) , \partial U) \to 0$ as $t\to c$.

\subsection{Uniformly quasiregular maps and dynamics}
\label{ss:uqr}

The maximal dilatation is sub-multiplicative under composition, that is, if $f: U\to U$ is quasiregular, then for every $m\geq 1$, we have $K(f^m) \leq K(f)^m$. If $K(f) >1$, then we cannot a priori assume a uniform bound on the maximal dilatation of the iterates. If there is such an upper bound $K$ on $K(f^m)$ over all $m\in \N$, then $f$ is called uniformly $K$-quasiregular. 

Every uniformly $K$-quasiregular map $f:U\to U$ has a measurable conformal structure $G$ on $U$. This is a measurable map $x \mapsto G(x)$ into symmetric positive-definite matrices with $\det G \equiv 1$, $||G||, ||G^{-1}|| \leq H(n,K)$ and $f^*G = G$ almost everywhere, where
\[ f^*G(x):= J_f(x)^{-2/n} Df(x)^T G(f(x)) Df(x).\]
See Iwaniec and Martin's paper \cite{IM96}, as well as their book \cite[Chapter 21]{IM01} for more on this idea; the construction, by averaging the pullbacks of the standard structure under the iterates, works verbatim for uniformly quasiregular self-maps of a domain $U$.

If $f:\S^n \to \S^n$ is uniformly quasiregular, then the definitions for the Fatou set $F(f)$ and the Julia set $J(f)$ go through verbatim from their definitions in complex dynamics. More precisely, $x\in F(f)$ if there is a neighbourhood $V$ of $x$ such that the family of iterates restricted to $V$, $\{ f^m|_V : m\in \N \}$, forms a normal family. On the other hand, $x\in J(f)$ if there is no such neighbourhood. Evidently $\S^n$ is then the disjoint union of $F(f)$ and $J(f)$, but less trivially, if $f$ is non-injective, then $J(f)$ is non-empty. We refer to \cite{FN} for much more on the foundations of quasiregular dynamics. Following \cite[Section 4]{HMM}, a component $U$ of $F(f)$ is called {\it stable} if $f(U) \subset U$; a stable component is an {\it attracting} (or {\it super-attracting}) {\it basin} if $f^m \to x_0 \in U$ locally uniformly on $U$ with $x_0$ an attracting (super-attracting) fixed point, a {\it rotation domain} if $f|_U$ is injective and some subsequence $f^{m_j} \to \operatorname{id}$ locally uniformly on $U$, and a {\it parabolic basin} if some subsequence $f^{m_j}$ converges locally uniformly on $U$ to a fixed point of $f$ in $\partial U$. We refer to \cite{HMM} for the precise statements.

\subsection{Metrics}
\label{ss:metrics}

For a proper subdomain $U$ of $\R^n$, the modulus metric is defined by
\[ \mu_U(x,y) = \inf_C M( \Delta (C , \partial U ; U)),\]
where the infimum is taken over continua $C \subset U$ containing $x$ and $y$. If $\operatorname{cap} \partial U >0$, then $\mu_U$ is a metric, see \cite[Section 2]{Vuo85} or \cite[Section 10.2]{HKV}, and if $U=\B^n$, then
\[ \mu_{\B^n} (x,y) = \Phi (\rho(x,y)),\]
where $\rho$ is the hyperbolic metric on $\B^n$, and 
\[ \Phi(t) = \gamma_n ( 1/\tanh(t/2))\]
where $\gamma_n$ is the Gr\"otzsch capacity function. See \cite[Section 10.2]{HKV} for these definitions and properties. In particular, $\Phi$ is an increasing homeomorphism of $[0,\infty)$ onto itself. We will need the following estimate of Anderson \cite{And74}, see also \cite[Theorem 9.17]{HKV}. For all $n\geq 2$ and $s>1$, there exists a constant $c_n>0$ depending only on $n$ such that
\begin{equation}
\label{eq:cn} 
2^{n-1} c_n \log \frac{s+1}{s-1} \leq \gamma_n(s) \leq 2^{n-1}c_n \log \frac{4(s+1)}{s-1}.
\end{equation}

Assuming still that $U$ is a proper subdomain of $\R^n$, let $k_U$ be the quasihyperbolic metric of $U$, that is, the conformal metric obtained by integrating the density $\dist(x,\partial U)^{-1}$, and let
\[ j_U(x,y) = \log \left ( 1+ \frac{ |x-y|}{ \min \{ \dist (x,\partial U) , \dist(y,\partial U) \} } \right ).\]
Gehring and Osgood \cite{GO79} defined a uniform domain $U$ as one for which $k_U \leq cj_U +d$. We collect several results in the following lemma.

\begin{lemma}
\label{lem:list}
\begin{enumerate}[(a)]
\item There exist constants $b_1,b_2$ depending only on $n$ such that $\mu_U \leq b_1 k_U + b_2$ for every such $U$, see \cite[Lemma 2.39(3)]{Vuo85} or \cite[Lemma 10.7]{HKV}.
\item If $\S^n \setminus U$ satisfies a certain metric thickness condition, see \cite[(2.21)]{Vuo87}, then $\mu_U \geq a_1 j_U$, with $j_U$ taken in the spherical metric and $a_1$ depending on $n$ and the thickness parameters, see \cite[Theorem 2.27]{Vuo87}. In particular, $\mu_U$ is complete. A sufficient condition for the metric thickness property is that $\partial U$ (viewed as a subset of $\S^n$) is connected and contains at least two points. More generally, if $\partial U$ is uniformly perfect, then it satisfies the thickness condition, see J\"arvi and Vuorinen \cite[Theorem 4.4]{JV96}; uniformly perfect sets have positive capacity, see \cite[Theorem 4.1 and Corollary 4.2]{JV96}.
\item The Julia set of a uniformly quasiregular map $f:\S^n \to \S^n$ of degree at least $2$ is uniformly perfect, see the paper of the author and Nicks \cite{FN11}.
\end{enumerate}
\end{lemma}

For a metric $d$ on a set $X$ and $o\in X$, the Gromov product is
\[ (u|w)_o = \tfrac12 \left ( d(o,u) + d(o,w) - d(u,w) \right ),\]
and we say that $(X,d)$ is Gromov hyperbolic in the four-point sense if there exists $\delta>0$ with
\[ (u|w)_o \geq \min \{ (u|v)_o , (v|w)_o \} - \delta\]
for all $u,v,w,o\in X$. A sequence $(z_k)$ with $d(o,z_k) \to \infty$ and $(z_k|z_m)_o \to \infty$ as $k,m\to\infty$ is called a Gromov sequence, two Gromov sequences $(z_k),(w_k)$ being equivalent if $(z_k|w_k)_o \to \infty$; the Gromov boundary is the set of equivalence classes. For $X = \B^n$, if in addition we have 
\begin{equation}
\label{eq:gromov}
d(o,z_k) \to \infty, \quad (z_k|z_m)_o \to \infty, \quad \text{ as } k,m\to \infty, 
\end{equation}
then the sequence $(z_m)$ converges in the Euclidean topology to a point of $\S^{n-1}$, equivalent sequences have the same limit, every point of $\S^{n-1}$ arises, and inequivalent sequences have distinct limits, we say that the Gromov boundary of $(\B^n,d)$ is $\S^{n-1}$ in the usual way. For the hyperbolic metric this is classical, see \cite[Chapter III.H.3]{BH99}, and one may take $\delta = \log 3$, see \cite{Lin07} and the references there.

Finally in this section, we will need the following result of Ca\l{}ka.

\begin{theorem}[Theorem 5.6, \cite{Cal84}]
\label{thm:calka}
Let $f$ be a nonexpansive mapping of a finitely totally bounded metric space $M$ into itself. If for some $z_0\in M$ the sequence $\{ f^m(z_0) \}$ contains a bounded subsequence then for every $z\in M$ the sequence $\{f^m(z) \}$ is bounded.
\end{theorem}

We will apply this result to proper metric spaces (that is, closed bounded sets are compact) and \cite[Proposition 2.7(i)]{Cal84} shows that this implies finitely totally bounded.

\section{The $G$-modulus}

Here we define the $G$-modulus of a path family.
Let $U$ be a proper subdomain of $\R^n$ and let $G$ be a bounded measurable conformal structure on $U$ and set $\lambda = H^{1/2}$, recalling Section \ref{ss:uqr}. As the eigenvalues of $G(x)$ lie in $[H^{-1},H]$, we have
\begin{equation}
\label{eq:vcomp} 
\lambda^{-1} |v| \leq |v|_{G(x)} := \langle G(x)v,v \rangle ^{1/2} \leq \lambda |v| 
\end{equation}
for almost every $x\in U$ and all $v$. Since line integrals along paths see sets of measure zero, we fix a Borel representative $\tilde{G}$ of $G$, that is, a Borel map agreeing with $G$ almost everywhere; we may and do assume that \eqref{eq:vcomp} holds for $\tilde{G}(x)$ at every $x$. The following elementary fact shows that the choice does not matter.

\begin{lemma}
\label{lem:null}
If $N\subset \R^n$ has measure zero, then the family $\Gamma_N$ of locally rectifiable paths meeting $N$ in a set of positive length has modulus zero.
\end{lemma}

\begin{proof}
Enlarge $N$ to a Borel null set. For $k\in \N$ let $\Gamma_{N,k}$ consist of the paths meeting $N$ in length at least $1/k$. The Borel density $k \chi_N$ is admissible for $\Gamma_{N,k}$ with $\int (k\chi_N)^n \: dm = 0$, so $M(\Gamma_{N,k}) = 0$, and $M(\Gamma_N) = 0$ by countable subadditivity.
\end{proof}

For a locally rectifiable path $\gamma$, parametrized by Euclidean arclength, and a Borel $\eta \geq 0$, set
\[ \int_{\gamma} \eta \: |d\gamma|_G : = \int_0^{\ell(\gamma)} \eta(\gamma(s)) |\gamma'(s)|_{\tilde{G}(\gamma(s))} \: ds,\]
where $\ell(\gamma)$ denotes the length of $\gamma$.

\begin{definition}
\label{def:Gadmiss}
We say that $\eta$ is $G$-admissible for the path family $\Gamma$ if there is a subfamily $\Gamma_0 \subset \Gamma$ with $M(\Gamma_0) = 0$ such that $\int_{\gamma} \eta \: |d\gamma|_G \geq 1$ for every locally rectifiable $\gamma \in \Gamma \setminus \Gamma_0$. We then define the $G$-modulus of $\Gamma$ to be
\[ M_G(\Gamma) = \inf \int_U \eta^n \: dm \]
where the infimum is over all $G$-admissible $\eta$.
\end{definition}

We give some important properties of the $G$-modulus.

\begin{lemma}
\label{lem:MGind}
The $G$-modulus $M_G$ does not depend on the choice of Borel representative $\tilde{G}$.
\end{lemma}

\begin{proof}
Suppose that $\tilde{G_1}$ and $\tilde{G_2}$ are two Borel representatives of $G$ that agree outside a null set $N$. By Lemma \ref{lem:null}, $M(\Gamma_N) = 0$, and off $\Gamma_N$ the two integrals $\int_{\gamma} \eta |d\gamma|_{G_i}$ coincide for every $\eta$. Hence the two notions of $G$-admissibility agree.
\end{proof}

\begin{lemma}
\label{lem:Gcomp}
For every path family $\Gamma$ in $U$, we have
\[ \lambda^{-n} M(\Gamma) \leq M_G(\Gamma) \leq \lambda^n M(\Gamma).\]
In particular, $M_G$-null families are exactly $M$-null families, and $M_G$ is countably subadditive and satisfies the minorization inequality.
\end{lemma}

\begin{proof}
If $\eta$ is $G$-admissible, then $\lambda \eta$ is admissible in the classical sense for $\Gamma \setminus \Gamma_0$, since by the convention preceding Lemma \ref{lem:null}, \eqref{eq:vcomp} holds for $\tilde{G}$ everywhere. Since $M(\Gamma) = M(\Gamma\setminus \Gamma_0)$, this gives $M(\Gamma) \leq \lambda^n M_G(\Gamma)$. Symmetrically, we have $M_G(\Gamma) \leq \lambda^n M(\Gamma)$. The remaining properties have the same proofs as for $M$, as exceptional families are absorbed into $\Gamma_0$.
\end{proof}

\section{The invariant modulus metric}

In this section, we define the invariant modulus metric and study its properties. Throughout, $U$ is a proper subdomain of $\R^n$, for $n\geq 2$, with $\operatorname{cap} \partial U >0$, so that $\mu_U$ is a metric.

\begin{definition}
\label{def:Gmetric}
For $x,y\in U$, we set
\[ d_G(x,y) = \inf \{ M_G(\Delta ( C , \partial U ; U)) : C\subset U \text{ a continuum }, x,y \in C \}.\]
\end{definition}

Immediately from Lemma \ref{lem:Gcomp} and Definition \ref{def:Gmetric}, we see that
\begin{equation}
\label{eq:mucomps}
\lambda^{-n} \mu_U \leq d_G \leq \lambda^n \mu_U.
\end{equation}

\begin{proposition}
\label{prop:dGmetric}
The expression $d_G$ in Definition \ref{def:Gmetric} is a metric on $U$.
\end{proposition}

\begin{proof}
Symmetry is clear. For positivity, if $x\neq y$, then $d_G(x,y) \geq \lambda^{-n} \mu_U(x,y) > 0$ as $\mu_U$ is a metric. If we take $C = \{x \}$, then $M( \Delta ( \{ x \} , \partial U ; U)) =0$ as we have a family of paths all passing through one point. Then Lemma \ref{lem:Gcomp} implies that $M_G(\Delta ( \{ x \} , \partial U ; U)) =0$ and hence $d_G(x,x) = 0$.

Finally, for the triangle inequality, given $\epsilon >0$, choose continua $C_1 \ni x,y$ and $C_2 \ni y,z$ with $M_G(\Delta(C_i,\partial U ; U))$ within $\epsilon$ of the respective infima. Then $C =C_1 \cup C_2$ is a continuum containing $x,z$ and a path starting in $C$ starts in $C_1$ or in $C_2$, so 
\[ \Delta ( C , \partial U ; U) = \Delta (C_1,\partial U ; U) \cup \Delta ( C_2 , \partial U ; U).\] 
Then subadditivity gives 
\[ d_G(x,z) \leq d_G(x,y) + d_G(y,z) + 2\epsilon.\]
As $\epsilon $ was arbitrary, the triangle inequality holds.
\end{proof}

For the rest of this section, we will set $U = \B^n$.

\begin{proposition}
\label{prop:top}
We have the following properties.
\begin{enumerate}[(a)]
\item The metric $d_G$ induces the Euclidean topology on $\B^n$. 
\item A set is $d_G$-bounded if and only if it is $\rho$-bounded if and only if it is relatively compact in $\B^n$.
\item The metric space $(\B^n,d_G)$ is proper.
\item If $x\in \B^n$, then $d_G(x,y_j) \to \infty$ if and only if $\dist(y_j , \S^{n-1}) \to 0$.
\end{enumerate}
\end{proposition}

\begin{proof}
By \eqref{eq:mucomps} and Section \ref{ss:metrics}, we have
\[ \lambda^{-n} \Phi ( \rho(x,y)) \leq d_G(x,y) \leq \lambda^n \Phi (\rho(x,y)),\]
where $\Phi$ is an increasing homeomorphism of $[0,\infty)$. All the claims now follow from the corresponding ones for the hyperbolic metric $\rho$: it induces the Euclidean topology, $\rho$-bounded sets are exactly the relatively compact subsets of $\B^n$, $(\B^n,\rho)$ is proper, and $\rho(x,y_j) \to \infty$ if and only if $\dist (y_j , \S^{n-1}) \to 0$.
\end{proof}

\begin{lemma}
\label{lem:roughhyp}
For all $x,y\in \B^n$, we have
\[ \lambda^{-n} 2^{n-1} c_n \rho(x,y) \leq d_G(x,y) \leq \lambda^n 2^{n-1}c_n ( \rho(x,y) + \log 4),\]
where $c_n$ is the constant from \eqref{eq:cn} that depends only on $n$.
\end{lemma}

\begin{proof}
Put $t = \rho(x,y)$ and $s = 1/\tanh(t/2)$ so that $\mu_{\B^n}(x,y) = \gamma_n(s)$ by Section \ref{ss:metrics}. Then
\[ \frac{s+1}{s-1} = \frac{ \cosh(t/2) + \sinh(t/2) }{\cosh(t/2) - \sinh(t/2)} = e^t.\]
Then Anderson's result in \eqref{eq:cn} gives
\[ 2^{n-1} c_n\rho(x,y) \leq \mu_{\B^n}(x,y) \leq 2^{n-1} c_n \rho(x,y) + 2^{n-1}c_n \log 4.\]
The lemma now follows from \eqref{eq:mucomps}. The two-sided estimate for $\mu_{\B^n}$ is \cite[(2.24)]{Vuo85}, with $8\log 2$ in place of $\log 4$, and the identity $(s+1)/(s-1) = e^t$ is \cite[(4.2)]{Lin07}.
\end{proof}

It is worth pointing out that $\Phi(t) \to 0$ like $(\log (1/t) )^{1-n}$ as $t\to 0$, so $d_G$ is not bi-Lipschitz to Euclidean distance at small scales. Its induced length metric is just $+ \infty$, so $(\B^n , d_G)$ is inherently non-geodesic. A consequence of this is that Gromov hyperbolicity cannot be transported from $\rho$ to $d_G$ by quasi-isometry alone, and is why chain metrics will be introduced later to prove our Denjoy-Wolff Theorem. While we do not know if $(\B^n,d_G)$ is Gromov hyperbolic for general $G$, in the case $G\equiv I$ we have $d_G = \mu_{\B^n}$ and a result of Lind\'en \cite[Theorem 4.1]{Lin07} states that $(\B^n , \mu_{\B^n})$ is Gromov hyperbolic, with Gromov constant at most $2^{n-1}c_n \log 12$; his proof is the computation of Lemma \ref{lem:roughhyp} combined with the four-point condition for $\rho$. By conformal invariance his result covers all simply connected planar domains, and, since for $n=2$ the quasiconformal map $\varphi$ solving the Beltrami equation for $G$ satisfies $\varphi^*I = G$, so that $M_G(\Gamma) = M(\varphi(\Gamma))$ by Lemma \ref{lem:Gpol} applied to $\varphi$ and $\varphi^{-1}$ and hence $d_G$ is the pull-back of $\mu_{\B^2}$ under $\varphi$, it also shows that $(\B^2 , d_G)$ is Gromov hyperbolic for every bounded measurable $G$. We remark that Lind\'en speculates upon the possibility that $(U,\mu_U)$ is Gromov hyperbolic for general domains in higher dimensions, but observes that the problem is likely to be delicate. For example, H\"ast\"o \cite{Has06} showed that $j_U$ is not Gromov hyperbolic except in punctured space.

\section{Non-expansion and the Schwarz-Pick Lemma}

We start with a $G$ version of the Poletsky inequality.

\begin{lemma}
\label{lem:Gpol}
Let $f:U\to U'$ be non-constant quasiregular and $G,G'$ measurable conformal structures on $U,U'$ respectively with $f^*G' = G$ almost everywhere. Then for every path family $\Gamma$ in $U$, we have
\[ M_{G'}(f(\Gamma)) \leq M_G(\Gamma).\]
\end{lemma} 

\begin{proof}
The proof follows Rickman's proof of Poletsky's inequality \cite[Theorem II.8.1]{Rickman}, with the $K_I$-step replaced by an identity. We write $\ell(\gamma)$ for length, $\gamma^0$ for the arclength parametrization of a locally rectifiable path $\gamma$, and $s_{\gamma}$ for its length function, so that $\gamma = \gamma^0 \circ s_{\gamma}$.

{\it Step 1: null sets.} Let $B_f$ be the branch set of $f$. Then $|B_f| = |f(B_f)| = 0$, $f$ is differentiable almost everywhere with $J_f>0$ almost everywhere, and $f$ satisfies Lusin's condition (N), see \cite[Sections I.2 and I.4]{Rickman}. Let $N\subset U$ and $N'\subset U'$ be null sets outside which $\tilde{G} = G$ and $\tilde{G}' = G'$ respectively. Since $\int_{f^{-1}(N')} J_f \: dm \leq \int_{N'} N(y,f,U) \: dm(y) = 0$ by the area formula \cite[Lemma I.4.11]{Rickman} and $J_f>0$ almost everywhere, $f^{-1}(N')$ has measure zero. Let $E\subset U$ be a Borel null set containing $B_f$, $N$, $f^{-1}(N')$, the set where $f$ is not differentiable or $J_f = 0$, and the set where $f^*G' = G$ fails. Then $|E| = 0$ and, by condition (N), $|f(E)| = 0$.

{\it Step 2: the pointwise identity.} For $x\notin E$ and $v\in\R^n$, the invariance $Df(x)^TG'(f(x))Df(x)=J_f(x)^{2/n}G(x)$ gives
\[ |Df(x)v|_{G'(f(x))}^2 = \langle Df(x)^T G'(f(x)) Df(x) v , v \rangle = J_f(x)^{2/n} |v|_{G(x)}^2 .\]
Hence the minimal $G$-stretch
\[ \ell_G(Df(x)):=\min\{|Df(x)v|_{G'(f(x))}:|v|_{G(x)}=1\}\]
satisfies
\begin{equation}
\label{eq:lG}
|Df(x)v|_{G'(f(x))} = \ell_G(Df(x)) |v|_{G(x)} \text{ for all } v, \quad \text{and} \quad \ell_G(Df(x))^n=J_f(x), \qquad x\notin E .
\end{equation}

{\it Step 3: the pushed-forward density.} Let $\eta$ be $G$-admissible for $\Gamma$ with exceptional family $\Gamma_0$, and put $g = \eta \, J_f^{-1/n}$ on $U\setminus E$, a non-negative measurable function. Since $f$ is a local homeomorphism on the open set $U\setminus B_f$, which is second countable, there are open sets $A_1,A_2,\ldots$ covering $U\setminus B_f$ with $f|_{A_i}$ injective. Each $f|_{A_i\setminus E}$ is a continuous injection of a Borel set, so $f(A_i\setminus E)$ is Borel and $(f|_{A_i\setminus E})^{-1}$ is Borel measurable by the Lusin-Souslin theorem. Hence
\[ \eta_i'(y) = \begin{cases} g \big ( (f|_{A_i\setminus E})^{-1}(y) \big ), & y\in f(A_i \setminus E), \\ 0, & \text{otherwise}, \end{cases}\]
is measurable on $U'$, and so is
\[ \eta'(y) := \sup_i \eta_i'(y) = \sup \{ g(x) : x\in f^{-1}(y) \setminus E \} .\]
Choose a Borel function $\eta_0 \geq 0$ on $U'$ with $\eta_0 = \eta'$ outside a Borel null set $N''\subset U'$. Since $\eta'(y)^n \leq \sum_{x\in f^{-1}(y)\setminus E} g(x)^n$, the area formula \cite[Lemma I.4.11]{Rickman} and \eqref{eq:lG} give
\begin{equation}
\label{eq:mass}
\int_{U'} \eta_0^n \: dm = \int_{U'} (\eta')^n \: dm \leq \int_{U'} \sum_{x\in f^{-1}(y)\setminus E} g(x)^n \: dm(y) = \int_{U\setminus E} g^n J_f \: dm = \int_U \eta^n \: dm .
\end{equation}

{\it Step 4: exceptional families for $f(\Gamma)$.} Let $\Gamma_1\subset\Gamma$ be the family of paths $\gamma\in\Gamma$ such that $f\circ\gamma$ is locally rectifiable and $f$ is not absolutely precontinuous on $\gamma$ in the sense of \cite[Section II.5]{Rickman}; recall that this means that the path $\gamma^*$ defined by $\gamma = \gamma^* \circ s_{f\circ\gamma}$ and $f\circ \gamma^* = (f\circ\gamma)^0$ is not locally absolutely continuous; $\gamma^*$ is well defined because $f\circ\gamma$ is constant on an interval only if $\gamma$ is, $f$ being discrete. Poletsky's lemma \cite[Lemma II.5.3]{Rickman} states that $M(f(\Gamma_1)) = 0$. Next, $M(f(\Gamma_0)) \leq K_I(f) M(\Gamma_0) = 0$ by Poletsky's inequality. Finally, let $\Gamma_2'$ be the family of locally rectifiable paths in $U'$ meeting the null set $f(E) \cup N' \cup N''$ in positive length; $M(\Gamma_2') = 0$ by Lemma \ref{lem:null}. Put
\[ \Gamma_0' := f(\Gamma_0) \cup f(\Gamma_1) \cup \Gamma_2' ,\]
a family of modulus zero by countable subadditivity. We claim that $\int_{\beta} \eta_0 |d\beta|_{G'} \geq 1$ for every locally rectifiable $\beta \in f(\Gamma) \setminus \Gamma_0'$, which shows that $\eta_0$ is $G'$-admissible for $f(\Gamma)$; together with \eqref{eq:mass} and the infimum over $\eta$ this proves the lemma.

{\it Step 5: admissibility.} Let $\beta = f\circ \gamma$ be as in the claim, $\gamma\in\Gamma$. As $\beta\notin f(\Gamma_1)$, $f$ is absolutely precontinuous on $\gamma$, so $\gamma^*$ is locally absolutely continuous; in particular $\gamma$ is locally rectifiable, and $\gamma \notin \Gamma_0$ as $\beta \notin f(\Gamma_0)$. Since $\beta^0 = f\circ \gamma^*$ and $\beta \notin \Gamma_2'$, for almost every $t$ we have: $\gamma^*$ is differentiable at $t$, $\beta^0(t) \notin f(E)\cup N'\cup N''$, hence $\gamma^*(t) \notin E$, $f$ is differentiable at $\gamma^*(t)$ with
\[ (\beta^0)'(t) = Df(\gamma^*(t)) (\gamma^*)'(t), \qquad |(\beta^0)'(t)| = 1 ,\]
and $\eta_0(\beta^0(t)) = \eta'(\beta^0(t)) \geq g(\gamma^*(t))$. Using \eqref{eq:lG} at $x = \gamma^*(t)$,
\begin{align*} \int_{\beta} \eta_0 \: |d\beta|_{G'} &= \int_0^{\ell(\beta)} \eta_0(\beta^0(t)) |(\beta^0)'(t)|_{\tilde{G}'(\beta^0(t))} \: dt \geq \int_0^{\ell(\beta)} g(\gamma^*(t)) \, \ell_G ( Df (\gamma^*(t))) \, |(\gamma^*)'(t)|_{\tilde{G}(\gamma^*(t))} \: dt \\ &= \int_0^{\ell(\beta)} \eta(\gamma^*(t)) |(\gamma^*)'(t)|_{\tilde{G}(\gamma^*(t))} \: dt .\end{align*}
For a locally absolutely continuous path $\alpha$ and Borel $\eta\geq 0$ one has $\int_{\alpha} \eta \: |d\alpha|_G = \int \eta(\alpha(t)) |\alpha'(t)|_{\tilde{G}(\alpha(t))} \: dt$, by the change of variables $\alpha = \alpha^0\circ s_{\alpha}$ with $s_\alpha$ absolutely continuous, exactly as in \cite[Theorem 4.1]{Vaisala} for the Euclidean length element. Applied to $\alpha = \gamma^*$, and since $\gamma = \gamma^*\circ s_{f\circ\gamma}$ with $s_{f\circ\gamma}$ continuous and non-decreasing, so that $\gamma$ and $\gamma^*$ have the same arclength parametrization, the last integral equals $\int_{\gamma} \eta \: |d\gamma|_G \geq 1$, because $\gamma\in\Gamma\setminus\Gamma_0$ is locally rectifiable. This proves the claim.
\end{proof}

\begin{lemma}
\label{lem:minlifts}
Let $f:U\to U$ be non-constant quasiregular and let $C\subset U$ be a continuum. Then $\Delta ( f(C) , \partial U ; U)$ is minorized by $f( \Delta ( C,\partial U ; U))$.
\end{lemma}

\begin{proof}
Let $\beta : [a,b) \to U$ lie in $\Delta ( f(C) , \partial U ; U)$ and pick $x_0 \in C$ with $f(x_0) = \beta (a)$. Let $\alpha :[a,c) \to U$ be a maximal $f$-lift of $\beta$ starting at $x_0$, using the path lifting property of Section \ref{ss:qrmaps}. If $c<b$, then $\dist (\alpha(t) , \partial U) \to 0$. If $c=b$, then also $\dist (\alpha(t) , \partial U) \to 0$, otherwise some $\alpha (t_j)$, where $t_j\to b$, stays in a compact subset $E$ of $U$ and then $\beta(t_j) = f(\alpha(t_j)) \in f(E)$, another compact subset of $U$. This contradicts $\dist (\beta(t) , \partial U) \to 0$. In both cases $\alpha \in \Delta(C,\partial U ; U)$ and $f\circ \alpha = \beta |_{[a,c)}$ is a subpath of $\beta$.
\end{proof}

We now come to our version of the Schwarz-Pick Lemma and proving Theorem \ref{thm:sp}. The argument for the main part is the same as Vuorinen's proof of $\mu_{f(U)} ( f(x),f(y) ) \leq K_I(f) \mu_U (x,y)$, see \cite[Lemma 3.1(1)]{Vuo85} or \cite[Theorem 15.36]{HKV}, which uses maximal liftings as in Lemma \ref{lem:minlifts}.

\begin{proof}[Proof of Theorem \ref{thm:sp}]
Given $\epsilon >0$, choose a continuum $C \ni x,y$ with $M_G( \Delta (C , \partial U ; U)) \leq d_G(x,y) + \epsilon$. Then $f(C)$ is a continuum containing $f(x),f(y)$, and by Lemmas \ref{lem:Gcomp}, \ref{lem:Gpol} and \ref{lem:minlifts} we have
\[ d_G(f(x),f(y)) \leq M_G ( \Delta ( f(C) , \partial U ; U)) \leq M_G( f( \Delta ( C,\partial U; U))) \leq M_G( \Delta (C,\partial U;U)) \leq d_G(x,y) + \epsilon.\]
As $\epsilon $ was arbitrary, the first claim follows. Note that $(f^m)^*G = G$ follows from $f^*G = G$ by the chain rule, so every iterate of $f$ is also $d_G$-non-expanding. For~(a), if $f$ is a homeomorphism, then $f^{-1}$ is quasiconformal, hence differentiable almost everywhere and satisfying Lusin's condition (N), so the identity $f^*G = G$ at $x$, transported to $y = f(x)$, gives $(f^{-1})^*G = G$ almost everywhere; both $f$ and $f^{-1}$ are then non-expanding by \eqref{eq:sp1}. For~(b), if $f(x) = f(y)$, then $d_G(x,y) = d_G(f(x),f(y)) = 0$ and hence $x=y$ and $f$ is injective. For~(c), we can take any $x\neq y$ with $f(x) = f(y)$ and obtain the conclusion.
\end{proof}

\section{Chain metrics: hyperbolicity without hyperbolicity of $d_G$}

In this section we construct the chain metric built from $d_G$ that we will show is Gromov hyperbolic. We set $a:=\lambda^n 2^{n-1} c_n$ and $r_1:= 3a \log 4$.

\begin{definition}
\label{def:chain}
Let $\Psi$ be the class of continuous non-decreasing functions $\psi : [0,\infty) \to [0, \infty)$ with $\psi(0)=0$ and $\psi(u) \geq u/2$ for all $u$. For $r>0$, an $r$-chain from $x$ to $y$ is a finite sequence $x=z_0,z_1,\ldots, z_m = y$ in $\B^n$ with $d_G(z_i,z_{i+1}) \leq r$ for all $i$. For $\psi \in \Psi$ put
\[ d_r^{\psi}(x,y) := \inf \left \{ \sum_{i=0}^{m-1} \psi \left ( d_G(z_i,z_{i+1}) \right ) : (z_i) \text{ is an $r$-chain from $x$ to $y$} \right \}. \]
\end{definition}

We will use two versions of a chain metric:
\begin{itemize}
\item the {\it plain chain metric} $d_r = d_r^{\operatorname{id}}$, that is, $\psi(u) = u$,
\item the {\it hyperbolized chain metric} $d_r^h:= d_r^{\psi_h}$ with
\[ \psi_h(u) := \max \{ 2^{n-1}c_n \Phi^{-1}(u) , u/2 \} ,\]
where $\Phi(t) = \gamma_n ( 1/ \tanh(t/2) )$ is the function from Section \ref{ss:metrics} so that $\mu_{\B^n} = \Phi \circ \rho$.
\end{itemize}

For brevity, we will write $d:= d_r^h$ with a fixed $r\geq r_1$ and call it {\it the chain metric}. The floor of $u/2$ in $\psi_h$ is there so that the H\"older-type distortion of $d_G$ at small scales doesn't make chains of very small steps have even smaller $d_G$-sums that collapse to $0$ in the infimum. We also point out that $r$-chains exist between any two points $x$ and $y$ in $\B^n$: just join $x$ to $y$ by a Euclidean path and subdivide it finely, using Proposition \ref{prop:top}~(a) to see that $d_G$ is continuous for the Euclidean topology.

We collect basic properties of the chain metric in the following lemma.

\begin{lemma}
\label{lem:chainprop}
Let $\psi \in \Psi$ and $r\geq r_1$.
\begin{enumerate}[(a)]
\item The expression $d_r^{\psi}$ is a metric, $d_r^{\psi} \geq \tfrac12 d_r \geq \tfrac12 d_G$, and $d_r^{\psi}(x,y) \leq \psi(d_G(x,y))$ whenever $d_G(x,y) \leq r$. For $\psi = \operatorname{id}$, $d_r = d_G$ on such pairs.
\item The metric $d_r^{\psi}$ induces the Euclidean topology, $(\B^n , d_r^{\psi})$ is proper, $d_r^{\psi}$-bounded sets are precisely the relatively compact subsets of $\B^n$.
\item We have 
\[ \tfrac12 \lambda^{-n} 2^{n-1}c_n \rho \leq d_r^{\psi} \leq \left ( \frac{ \rho }{\log 4} + 1 \right ) \psi(r).\]
\item If $g:\B^n \to \B^n$ is non-expanding for $d_G$, then $g$ is non-expanding for $d_r^{\psi}$. If $g$ is a bijective $d_G$-isometry, then it is a $d_r^{\psi}$-isometry.
\end{enumerate}
\end{lemma}

\begin{proof}
For~(a), symmetry and the triangle inequality come from reversing and concatenating chains. As $\psi(u) \geq u/2$, every chain sum is at least half the corresponding plain sum, which is at least $d_G(x,y)$ by the triangle inequality for $d_G$. This gives $d_r^{\psi} \geq \tfrac12 d_r \geq \tfrac12 d_G$, in particular positivity. When $d_G(x,y) \leq r$, the one-step chain gives the upper bound.

For~(b), by ~(a) and the continuity of $\psi$ at $0$, we have $d_r^{\psi}(x_j,x) \to 0$ if and only if $d_G(x_j,x) \to 0$, so the topologies agree. Since $d_r^{\psi} \geq \tfrac12 d_G$, the $d_r^{\psi}$-bounded sets are $d_G$-bounded and hence relatively compact by Proposition \ref{prop:top}~(b). Closed $d_r^{\psi}$-balls are closed in the Euclidean sense and relatively compact, hence compact.

For~(c), the lower bound is part~(a) and Lemma \ref{lem:roughhyp}. For the upper bound, let $t = \rho(x,y)$ and choose points $z_0 = x,\ldots, z_m=y$ on the $\rho$-geodesic from $x$ to $y$ with $\rho(z_i,z_{i+1}) \leq \log 4$ and $m\leq t/\log 4+1$. By Lemma \ref{lem:roughhyp}, 
\[ d_G (z_i,z_{i+1}) \leq a (\log 4 + \log 4) = 2a \log 4 \leq r,\]
so this is an $r$-chain and each step gives at most $\psi(r)$.

Finally, for~(d), if $(z_i)$ is an $r$-chain from $x$ to $y$ then $(g(z_i))$ is a chain from $g(x)$ to $g(y)$ with $d_G(g(z_i) , g(z_{i+1})) \leq d_G(z_i,z_{i+1}) \leq r$, and hence is an $r$-chain. Taking the infimum and applying this observation to $g$ and $g^{-1}$ gives the result.
\end{proof}

We conclude from Lemma \ref{lem:chainprop} that the identity map $\iota : (\B^n , d_r^{\psi}) \to (\B^n , \rho)$ is an $(L,C)$-quasi-isometry, that is, $L^{-1} d_r^{\psi} - C \leq \rho \leq L d_r^{\psi} + C$, with $L = \max \{ \psi(r)/\log 4, 2^{2-n}\lambda^n  / c_n , 1 \}$ and $C = \psi(r)$. Moreover, by Theorem \ref{thm:sp} and Lemma \ref{lem:chainprop}~(d), every uniformly quasiregular map $f:\B^n \to \B^n$ is $d_r^{\psi}$-non-expanding.

\begin{proposition}
\label{prop:roughgeos}
The space $(\B^n , d_r^{\psi})$ is $h$-roughly geodesic with $h = \psi(r) +1$. That is, for all $x,y$ there is a rough geodesic $\sigma :[0,S] \to \B^n$ with $\sigma(0)=x$, $\sigma(S) = y$, $S\leq d_r^{\psi}(x,y)+1$, and
\[ \left |d_r^{\psi} (\sigma(s) , \sigma(t) ) - |s-t| \right | \leq h,\]
for $s,t\in [0,S]$.
\end{proposition}

\begin{proof}
Write $d' = d_r^{\psi}$. Choose an $r$-chain $(z_i)_{i=0}^m$ from $x$ to $y$ with 
\[ S:= \sum_{i=0}^{m-1} \psi \left ( d_G(z_i,z_{i+1}) \right ) \leq d'(x,y) +1 ,\]
and let $S_i$ be the partial sums $0=S_0 \leq S_1 \leq \ldots \leq S_m =S$, with $S_{i+1} - S_i \leq \psi(r)$. After discarding repeated consecutive points we may assume $S_i < S_{i+1}$. Define $\sigma(s) = z_i$ for $s\in [S_i,S_{i+1})$ and $\sigma(S) = y$. For $i\leq j$ the sub-chain gives $d'(z_i,z_j) \leq S_j-S_i$, and
\[ d'(x,y) \leq d'(x,z_i) + d'(z_i,z_j) + d'(z_j,y) \leq S_i + d'(z_i,z_j) + (S-S_j) \]
gives $d'(z_i,z_j) \geq S_j-S_i-1$. If $s\in [S_i,S_{i+1})$, $t\in [S_j,S_{j+1})$, $i\leq j$, then $|t-s|$ is within $\psi(r)$ of $S_j-S_i$. Hence
\[ |d'(\sigma(s) , \sigma (t)) - |t-s|| \leq \psi(r)+1.\]
The endpoint $t=S$ is handled the same way.
\end{proof}

We now come to the main result of this section.

\begin{theorem}
\label{thm:gromovchain}
Let $\psi \in \Psi$ and $r\geq r_1$. Then $(\B^n , d_r^{\psi})$ satisfies the four-point Gromov hyperbolic condition, and a sequence $(z_k)$ is a Gromov sequence for $d_r^{\psi}$ if and only if it is one for $\rho$, with two Gromov sequences equivalent for $d_r^{\psi}$ if and only if they are equivalent for $\rho$. Consequently \eqref{eq:gromov} holds and the Gromov boundary of $(\B^n , d_r^{\psi})$ is $\S^{n-1}$ in the sense of Section \ref{ss:metrics}.
\end{theorem}

This theorem is a modest extension of \cite[Theorem 3.18]{Vai05} from intrinsic to roughly geodesic spaces. This result of V\"ais\"al\"a does not apply in our situation as $d_G$ does not contain short arcs between Euclidean-close points.

\begin{proof}
Write $X = (\B^n,d')$ with $d' = d_r^{\psi}$, $Y = (\B^n , \rho)$, $h = \psi(r)+1$, and let $\delta_0$ be a Gromov hyperbolicity constant for the geodesic space $Y$; for $Y = (\B^n,\rho)$ one may take $\delta_0 = \log 3$.

Let $\sigma:[0,S] \to X$ be an $h$-rough geodesic. Then $\iota \circ \sigma$ satisfies
\[ L^{-1} |s-t| - C' \leq \rho (\sigma(s),\sigma(t) ) \leq L|s-t| + C'\]
with $C' = C+Lh$, that is, it is an $(L,C')$-quasi-geodesic in $Y$, not necessarily continuous. By \cite[Theorem III.H.1.7]{BH99} (the stability of quasi-geodesics, sometimes called the Morse Lemma), its image lies within $\rho$-Hausdorff distance $H = H(L,C',\delta_0)$ of the $\rho$-geodesic joining its endpoints. Hence, given three points and $h$-rough geodesics $\sigma_1,\sigma_2,\sigma_3$ joining them pairwise, each $\iota(\sigma_i)$ lies in the $\rho$-neighbourhood of radius $\delta_0 + 2H$ of $\iota(\sigma_j)\cup \iota ( \sigma_k)$, because the corresponding $\rho$-geodesic triangle is $\delta_0$-slim. Pulling back through the quasi-isometry, each $\sigma_i$ lies in the $d'$-neighbourhood of radius $\delta_1:= L(\delta_0 + 2H) +C$ of $\sigma_j\cup \sigma_k$. So rough geodesic triangles in $X$ are slim.

Let $w\in X$ and let $\sigma$ be an $h$-rough geodesic from $x$ to $y$. Then we claim that
\begin{equation}
\label{eq:gromovchain1}
(x|y)_w - h -1 \leq d'(w,\sigma) \leq (x|y)_w + \kappa, \quad \kappa = 2\delta_1+8h .
\end{equation}
For the lower bound, for $p=\sigma(s)$,
\[ d'(x,p) + d'(p,y) \leq S+2h \leq d'(x,y) + 2h+1,\]
and adding $d'(w,p) \geq d'(w,x) - d'(x,p)$ to $d'(w,p) \geq d'(w,y) - d'(p,y)$ gives $d'(w,p) \geq (x|y)_w - h -\tfrac12$. For the upper bound, let $\sigma_{xw}, \sigma_{wy}$ be $h$-rough geodesics and choose $p\in \sigma$ with $|d'(x,p) - (y|w)_x| \leq 2h$. This is possible as $s\mapsto d'(x,\sigma(s))$ is within $h$ of $s$ and $S \geq d'(x,y) - h$. As rough geodesic triangles in $X$ are slim, $p$ is within $\delta_1$ of some $q$ on $\sigma_{xw}$ or on $\sigma_{wy}$. In the first case, the rough geodesic property of $\sigma_{xw}$ gives 
\[ d'(w,q) \leq d'(x,w) - d'(x,q) +3h \leq d'(x,w) - d'(x,p) +\delta_1 +3h \leq d'(x,w) - (y|w)_x + \delta_1 +5h = (x|y)_w + \delta_1+5h,\]
using $d'(x,w) - (y|w)_x = (x|y)_w$. So $d'(w,p) \leq (x|y)_w + 2\delta_1 + 5h$. In the second case, $q\in \sigma_{wy}$, the same argument with the roles of $x$ and $y$ exchanged uses $d'(y,p) \geq S - s - h \geq d'(x,y) - s - 2h \geq (x|w)_y - 5h$, where $s \leq d'(x,p) + h \leq (y|w)_x + 3h$, and $d'(y,w) - (x|w)_y = (x|y)_w$, giving $d'(w,p) \leq (x|y)_w + 2\delta_1 + 8h$.

Next, let $x,y,z,w \in X$ and choose $p$ on an $h$-rough geodesic $\sigma_{xz}$ with $d'(w,p) \leq (x|z)_w + \kappa+1$ by \eqref{eq:gromovchain1}. Then $p$ is within $\delta_1$ of some $q \in \sigma_{xy} \cup \sigma_{yz}$, say $q\in \sigma_{xy}$. Then by \eqref{eq:gromovchain1},
\[ (x|y)_w \leq d'(w,\sigma_{xy}) + h+1 \leq d'(w,q) + h +1 \leq d'(w,p) +\delta_1 + h+1 \leq (x|z)_w + \kappa +\delta_1 +h+2.\] 
The case $q\in \sigma_{yz}$ gives the same bound for $(y|z)_w$. Hence
\[ (x|z)_w \geq \min \{ (x|y)_w , (y|z)_w \} - \delta, \quad \delta :=\kappa+\delta_1 + h+2,\]
showing the four-point Gromov hyperbolicity condition.

Now, let $(z_k)$ be any sequence in $\B^n$, let $\sigma_{km}$ be an $h$-rough geodesic in $X$ from $z_k$ to $z_m$ and $[z_k,z_m]$ the $\rho$-geodesic. In $Y$, the geodesic case $h=0$, $\delta_1 = \delta_0$ of \eqref{eq:gromovchain1} gives
\[ \rho ( o, [z_k,z_m]) \geq (z_k|z_m)_o^{\rho} \geq \rho ( o, [z_k,z_m]) - 2\delta_0.\]
By the stability of quasi-geodesics, the $\rho$-Hausdorff distance between $[z_k,z_m]$ and $\iota(\sigma_{km})$ is at most $H$, so 
\[ | \rho ( o , [z_k,z_m]) - \rho ( o ,\sigma_{km} ) | \leq H.\]
By the quasi-isometry inequalities and \eqref{eq:gromovchain1}, we have
\[ \frac{ d'(o,\sigma_{km}) - C}{L} \leq \rho(o,\sigma_{km}) \leq L d'(o,\sigma_{km}) + C, \qquad (z_k|z_m)_o^{d'} - h - 1 \leq d'(o,\sigma_{km}) \leq (z_k|z_m)_o^{d'} + \kappa .\]
Combining the three displays,
\[ \frac{ (z_k|z_m)_o^{d'} - h - 1 - C }{L} - H - 2\delta_0 \leq (z_k|z_m)_o^{\rho} \leq L (z_k|z_m)_o^{d'} + L\kappa + C + H ,\]
and $d'(o,z_k) \to \infty$ if and only if $\rho(o,z_k) \to \infty$. Hence $(z_k)$ is a Gromov sequence for $d'$ if and only if it is one for $\rho$, and applying the same estimate to the pairs $(z_k,w_k)$ shows that two Gromov sequences are equivalent for $d'$ if and only if they are equivalent for $\rho$. The remaining claims now follow from the corresponding classical facts for $(\B^n,\rho)$ recalled in Section \ref{ss:metrics}.
\end{proof}

The proof requires that $Y = (\B^n , \rho)$ is a proper geodesic Gromov hyperbolic space, quasi-isometric to $(\B^n , d_r^{\psi})$ via the identity, whose Gromov sequences converge in $\overline{\B^n}$ to points of $\S^{n-1}$ with equivalent sequences having the same limit. Any proper geodesic Gromov hyperbolic metric on a domain $U$ with these properties can replace it. This will be our strategy when we come to Theorem \ref{thm:stable}.

Next, we point out that for the standard structure, the hyperbolized chain metric recovers the hyperbolic metric.

\begin{proposition}
\label{prop:hypchain}
Let $G \equiv I$ and $r\geq r_1$. Then $d_r^h \geq 2^{n-1}c_n \rho$ everywhere and, whenever $\rho(x,y) \geq \log 4$,
\[ d_r^h(x,y) = 2^{n-1}c_n \rho(x,y).\]
\end{proposition}

\begin{proof}
For $G\equiv I$, we have $d_G = \Phi \circ \rho$. Every step of the chain gives 
\[ \psi_h ( \Phi (\rho_i)) \geq 2^{n-1}c_n \Phi^{-1} ( \Phi ( \rho_i)) = 2^{n-1}c_n \rho_i,\]
so every chain has sum at least $2^{n-1}c_n \sum_i \rho_i \geq 2^{n-1}c_n \rho(x,y)$. For the reverse inequality, first note that $\psi_h(\Phi (s)) = 2^{n-1}c_n s$ as soon as $2^{n-1}c_n s \geq \Phi(s)/2$. By Lemma \ref{lem:roughhyp}, this holds for $s\geq \log 4$. Put $s_r = \Phi^{-1}(r)$. As $\Phi(2\log 4) \leq 3\cdot 2^{n-1}c_n \log 4 \leq r$, we have $s_r \geq 2\log 4$. Let $t = \rho(x,y) \geq \log 4$. If $t\leq s_r$, the one-step chain gives $\psi_h(\Phi(t)) = 2^{n-1}c_n t$. If $t>s_r$, take $m = \lceil t/s_r \rceil \leq 2t/s_r$ equal pieces along the $\rho$-geodesic of length $\ell = t/m \in [s_r/2,s_r] \subset [\log 4, s_r]$. Each is an admissible step yielding $2^{n-1}c_n \ell$ and the total is $2^{n-1}c_n t$.
\end{proof}

On the other hand, the plain chain metric is not roughly isometric to $d_G$.

\begin{proposition}
\label{prop:notrough}
Let $G\equiv I$, so $d_G = \mu_{\B^n} = \Phi \circ \rho$, and write $\Phi(t) = 2^{n-1}c_n t + E(t)$, where $0\leq E \leq 2^{n-1}c_n\log 4$ by Lemma \ref{lem:roughhyp}. For $r\geq r_1$, put $s_r = \Phi^{-1}(r)$ and $e_r: = \inf_{0<s\leq s_r} E(s)/s \geq 0$. Then
\[ d_r \geq (2^{n-1}c_n + e_r) \rho, \quad \text{while} \quad \mu_{\B^n} \leq 2^{n-1}c_n (\rho + \log 4).\]
Consequently, if $E>0$ on $(0,s_r]$, then $e_r>0$ and $d_r - \mu_{\B^n} \to \infty$ along any sequence with $\rho \to \infty$. Thus the metrics $d_r$ and $d_G$, which are always quasi-isometric, are not roughly isometric whenever the lower bound in \eqref{eq:cn} is strict at $s = 1/\tanh(t/2)$ for all $0<t\leq s_r$.
\end{proposition}

\begin{proof}
To start, we have $d_G(z_i,z_{i+1}) \leq r$ if and only if $\rho_i := \rho(z_i,z_{i+1}) \leq \Phi^{-1}(r) = s_r$. This justifies taking the infimum of $E(s)/s$ over only $(0,s_r]$. A chain from $x$ to $y$ with $\rho$-steps $\rho_i \leq s_r$ yields 
\[ \sum_i \Phi (\rho_i) = \sum_i (2^{n-1}c_n + E(\rho_i) / \rho_i ) \rho_i \geq (2^{n-1}c_n + e_r) \sum_i \rho_i \geq (2^{n-1}c_n + e_r) \rho(x,y).\]
Taking the infimum over chains gives $d_r \geq (2^{n-1}c_n + e_r)\rho$.

To justify why $e_r >0$, write $E(s)/s = \Phi(s)/s - 2^{n-1}c_n$. As $\Phi(s)$ is comparable to $(\log (1/s) )^{1-n}$ as $s\to 0$, we have $\Phi(s) / s \to \infty$ and $E(s)/s > 2^{n-1}c_n$ for $s<\epsilon$ for some $\epsilon >0$. On $[\epsilon, s_r]$, $E(s)/s$ is continuous and positive by the hypotheses. We conclude that $e_r >0$. Finally,
\[ d_r - \mu_{\B^n} \geq (2^{n-1}c_n + e_r) \rho - 2^{n-1}c_n \rho - 2^{n-1}c_n \log 4 = e_r \rho - 2^{n-1}c_n \log 4 \to \infty\]
as $\rho \to \infty$, so $d_r$ and $\mu_{\B^n}$ are not roughly isometric.
\end{proof}

We remark that different admissible $r$ and $\psi$ will give bi-Lipschitz equivalent metrics, which are all non-expanded by $f$, all Gromov hyperbolic and all satisfy \eqref{eq:gromov}. Indeed, for $\psi,\psi'\in\Psi$ continuous, $\psi(u)/\psi'(u)$ is bounded above and below by positive constants on $(0,r]$ when $\psi(u)/u$ and $\psi'(u)/u$ are, which is the case for $\psi = \operatorname{id}$ and $\psi = \psi_h$, so $d_r^{\psi}$ and $d_r^{\psi'}$ are bi-Lipschitz equivalent; and for $r_1\leq r < r'$ we have $d_{r'}^{\psi} \leq d_r^{\psi}$, while a step with $r < d_G(z,z') \leq r'$ can be replaced by an $r$-chain of cost at most $(\rho(z,z')/\log 4 + 1)\psi(r) \leq C(r,r') d_G(z,z')$ by Lemma \ref{lem:chainprop}~(c) and Lemma \ref{lem:roughhyp}, so $d_r^{\psi} \leq C(r,r') d_{r'}^{\psi}$. Bi-Lipschitz changes of $d$ do not change the Denjoy-Wolff point we establish in Theorem \ref{thm:dw}, nor the positive-step/zero-step dichotomy, since the step limit $s$ changes by at most a bounded factor.

\section{The Denjoy-Wolff Theorem}

In this section, we will prove Theorem \ref{thm:dw}. Throughout this section, $f:\B^n \to \B^n$ is a uniformly quasiregular map with invariant measurable conformal structure $G$, and $d =d_r^h$ is the chain metric with a fixed $r\geq r_1$. We fix $x\in \B^n$, set $x_m = f^m(x)$, $a_m = d(x,x_m)$, and $s_m = d(x_m,x_{m+1})$.

\begin{proposition}
\label{prop:dicho}
Every orbit is either relatively compact in $\B^n$, or $\dist ( f^m(z) , \S^{n-1}) \to 0$ for every $z\in \B^n$.
\end{proposition}

\begin{proof}
Applying Theorem \ref{thm:calka} to the proper space $(\B^n,d)$ and the non-expanding map $f$, either all orbits are $d$-bounded or every orbit leaves every $d$-bounded set. Lemma \ref{lem:chainprop}~(b) then finishes the proof.
\end{proof}

Next, we essentially just restate some of the results of Hinkkanen, Martin and Mayer \cite{HMM}.

\begin{proposition}
\label{prop:hmm}
In the bounded case, either $f^m \to x_0 \in \B^n$ locally uniformly, where $x_0$ is the unique fixed point of $f$ in $\B^n$ and is attracting or super-attracting, or $f$ is a uniformly quasiconformal homeomorphism of $\B^n$ onto itself, a $d$-isometry, and some subsequence $f^{m_j} \to \operatorname{id}$ locally uniformly.
\end{proposition}

\begin{proof}
Every locally uniform limit of a subsequence of $(f^m)$ maps $\B^n$ into $\B^n$ as the orbits are relatively compact. Thus the work of \cite{HMM} applies. If some limit is a constant $x_0$, then $x_0 \in \B^n$ is a (super-)attracting fixed point by \cite[Proposition 4.6]{HMM} and $f^m\to x_0$ locally uniformly. This rules out the existence of a second fixed point in $\B^n$.

If all limits are non-constant, then \cite[Proposition 4.9]{HMM} shows that $f$ is a homeomorphism of $\B^n$ onto itself and some $f^{m_j} \to \operatorname{id}$. This latter fact is because the limit functions form a compact abelian group, which is a Lie group by \cite[Theorem 2.1]{Mar02}. In the second case, $f$ is a $d_G$-isometry by Theorem \ref{thm:sp}~(a) and a $d$-isometry by Lemma \ref{lem:chainprop}~(d). 
\end{proof}

\begin{proposition}
\label{prop:escape}
Assume that $\dist ( f^m(x) , \S^{n-1} ) \to 0$ as $m\to \infty$. Then:
\begin{enumerate}[(a)]
\item $a_m \to \infty$ and $|x_m - x_{m+1}| \to 0$.
\item For every $z\in \B^n$, $|f^m(x) - f^m(z)| \to 0$.
\item The set $\Lambda$ of Euclidean accumulation points of $(x_m)$ is a non-empty continuum in $\S^{n-1}$, and it is the same for every orbit.
\item We have $\dist ( f^m(E) , \Lambda) \to 0$ for every $d$-bounded set $E$. In particular, every locally uniform limit of a subsequence of $(f^m)$ is a constant in $\Lambda$, and every point of $\Lambda$ arises in this way. 
\item $s = \lim_k s_k$ exists, and $0\leq s \leq d(x,f(x))$.
\end{enumerate}
\end{proposition}

\begin{proof}
For~(a), we have $a_m\to \infty$ by Lemma \ref{lem:chainprop}~(b). Since $s_m\leq s_0$, Lemma \ref{lem:chainprop}~(c) gives $\rho(x_m,x_{m+1}) \leq R_0:= 2^{2-n}\lambda^ns_0/c_n$. A hyperbolic ball of fixed radius centred at points tending to $\S^{n-1}$ has Euclidean diameter tending to $0$, so $|x_m - x_{m+1}| \to 0$. For~(b), we run the same argument with $d(f^m(x),f^m(z)) \leq d(x,z)$.

For~(c), a sequence in a compact metric space with $|x_m-x_{m+1}| \to 0$ has a connected set of accumulation points: if the accumulation set were split by two disjoint open sets, the sequence would eventually have to jump between them across a gap of fixed positive width. It is non-empty by the compactness of $\overline{\B^n}$ and lies in $\S^{n-1}$ by part~(a). For~(d), if $E$ is $d$-bounded, then $E \subset\overline{B_d(x,R)}$ for some $R>0$ and $f^m(E) \subset \overline{B_d(x_m,R)}$. The Euclidean diameter of this latter ball tends to $0$ as in part~(a). Finally, part~(e) follows from Theorem \ref{thm:sp} and Lemma \ref{lem:chainprop}~(d).
\end{proof}

Next we turn to horoball behaviour near the boundary.

\begin{proposition}
\label{prop:horo}
Assume $\dist (f^m(x) ,\S^{n-1} ) \to 0$. There is a $1$-Lipschitz function $h:\B^n \to \R$ with $h(x) = 0$, a locally uniform limit of functions $z\mapsto d(z,y_i) - d(x,y_i)$ with $\dist ( y_i , \S^{n-1}) \to 0$ (a horofunction of $(\B^n,d)$), such that $h(f^m(z)) \leq d(x,z)$ for all $z\in \B^n$ and $m\geq 0$. In particular, $h(f^m(x)) \leq 0$.
\end{proposition}

\begin{proof}
Since $a_m\to \infty$, the set $\{ m : a_m > \max _{j<m} a_j \}$ is infinite. Enumerate the set by $m_1<m_2 <\ldots$. Then for $1\leq k\leq m_i$,
\begin{equation}
\label{eq:horo1}
a_{m_i-k} \leq a_{m_i}.
\end{equation}
Put $y_i = x_{m_i}$ and $b_i(z) = d(z,y_i) - d(x,y_i)$. Each $b_i$ is $1$-Lipschitz with $|b_i(z)| \leq d(x,z)$, so by Arzel\`a-Ascoli and properness a subsequence converges locally uniformly to some $h$ with the stated properties. For fixed $m$ and $i$ large, non-expansion and \eqref{eq:horo1} gives
\[ b_i(f^m(z)) = d(f^m(z) , f^m(x_{m_i - m}) - a_{m_i} \leq d(z,x_{m_i-m}) - a_{m_i} \leq d(z,x) + a_{m_i-m} - a_{m_i} \leq d(x,z).\]
Now let $i\to \infty$ to complete the proof.
\end{proof}

Putting everything together, we proceed with the proof of the Denjoy-Wolff Theorem.

\begin{proof}[Proof of Theorem \ref{thm:dw}]
The metric claims are Lemma \ref{lem:chainprop} and Theorem \ref{thm:gromovchain}, the non-expansion is Theorem \ref{thm:sp} with Lemma \ref{lem:chainprop}~(d), the dichotomy is Proposition \ref{prop:dicho}, the bounded case is Proposition \ref{prop:hmm}, and in the escaping case the claims after "moreover" are Propositions \ref{prop:escape} and \ref{prop:horo}. It remains to show that $\Lambda = \{w \}$.

The key inequality is a restatement of \eqref{eq:horo1} in terms of Gromov products based at $x$. For $1\leq k \leq m_i$,
\begin{equation}
\label{eq:dw1}
2 ( x_k | y_i )_x = a_k +a_{m_i} - d(x_k , y_i) \geq a_k+a_{m_i} - a_{m_i-k} \geq a_k,
\end{equation}
where the first inequality is non-expansion, $d(x_k,x_{m_i}) = d(f^k(x),f^k(x_{m_i-k})) \leq d(x,x_{m_i-k})$. For $k,m\geq 1$, choose $i$ with $m_i \geq \max \{ k,m\}$. By Theorem \ref{thm:gromovchain} and \eqref{eq:dw1},
\[ (x_k|x_m)_x \geq \min \{ (x_k|y_i)_x , (y_i|x_m)_x \} - \delta \geq \tfrac12 \min \{ a_k,a_m \} - \delta \to \infty \]
as $k,m\to\infty$, as $a_k \to \infty$ by Proposition \ref{prop:escape}~(a). This is the argument of Karlsson \cite{Kar01} for non-expanding maps of Gromov hyperbolic spaces, as is Proposition \ref{prop:horo}.
By \eqref{eq:gromov} as established by Theorem \ref{thm:gromovchain}, $x_k$ converges in the Euclidean topology to a point $w\in \S^{n-1}$. Hence $\Lambda = \{ w \}$.

Finally, by Proposition \ref{prop:escape}~(b), $f^k(z) \to w$ for every $z\in \B^n$, and by Proposition \ref{prop:escape}~(d) the convergence is uniform on $d$-bounded sets, that is, on compact subsets of $\B^n$. This completes the proof.
\end{proof}

We close this section with a consistency check. Let $n=2$ and let $f:\D\to\D$ be uniformly quasiregular with invariant structure $G$. Then $G$ corresponds to a Beltrami coefficient $\mu_G$ with $\|\mu_G\|_\infty<1$, and if $\varphi:\D\to\D$ is the quasiconformal solution of the Beltrami equation $\varphi_{\overline{z}} = \mu_G\varphi_z$, then $\varphi^*I = G$ and $F = \varphi\circ f\circ\varphi^{-1}$ is $1$-quasiregular, that is, holomorphic; this is the classical way of seeing that uniformly quasiregular maps in the plane are quasiconformally conjugate to holomorphic maps, see \cite[Chapter 21]{IM01}. By Lemma \ref{lem:Gpol} applied to $\varphi$ and $\varphi^{-1}$ we have $M_G(\Gamma) = M(\varphi(\Gamma))$ for every path family, so $d_G(x,y) = \mu_{\D}(\varphi(x),\varphi(y)) = \Phi(\rho(\varphi(x),\varphi(y)))$. Thus Theorem \ref{thm:sp} for $f$ is the Schwarz-Pick Lemma for $F$, the chain metric is a bi-Lipschitz deformation of the pull-back of $\rho$ under $\varphi$, and Theorem \ref{thm:dw} is the classical Denjoy-Wolff Theorem for $F$, transported by $\varphi$; the bounded case with $f$ an automorphism corresponds to $F$ being an elliptic automorphism of $\D$.

\section{Stable Fatou components}

Throughout this section, $U$ is a proper subdomain of $\S^n$ and all metric notions are taken with respect to the spherical metric $\sigma$, with length element $2|dz|/(1+|z|^2)$, as in \cite[Chapter 7]{BHK}: $\sigma(x,\partial U)$ is the spherical distance to the boundary,
\[ \hat{\jmath}_U(x,y) = \log \left ( 1+ \frac{ \sigma(x,y)}{ \min \{ \sigma(x,\partial U) , \sigma(y,\partial U) \} } \right ),\]
and $\hat{k}_U$ is the spherical quasihyperbolic metric, obtained by integrating the density $\sigma(x,\partial U)^{-1}$ with respect to spherical arclength. We say that $U$ is {\it uniform} if $(U,\sigma)$ is an $A$-uniform metric space in the sense of \cite[Definition 1.9]{BHK}; it is locally compact, rectifiably connected and noncomplete, as that definition requires. A closed set $E\subset\S^n$ with at least two points is {\it $\alpha$-uniformly perfect}, $0<\alpha<1$, if for every $x\in E$ and every $0<r<\operatorname{diam}_\sigma E$ the closed annulus $\{ y : \alpha r \leq \sigma(x,y) \leq r \}$ meets $E$; this differs from the chordal definition of \cite{JV96} only in the constants. The spherical metric and the chordal metric $q$ used in \cite{JV96,Vuo87} are bi-Lipschitz equivalent, $q\leq \sigma \leq \tfrac{\pi}{2} q$, so it does not matter which of the two is used in these definitions; only the constants change. The spherical and Euclidean metrics define the same class of uniform domains in $\S^n$, and uniformity is invariant under M\"obius maps, see \cite[Theorem 2.10 and Remark 4.2.2]{Vai88}. If $U$ is bounded in some M\"obius chart, then in that chart the spherical and Euclidean metrics are bi-Lipschitz equivalent on $\overline{U}$, so $\hat{k}_U$, $\hat{\jmath}_U$ and uniform perfectness also agree, up to constants depending on the chart, with their Euclidean counterparts from Section \ref{ss:metrics}, and the reader may replace hats by their Euclidean versions throughout; compare the remarks at the start of \cite[Chapter 7]{BHK}. The spherical formulation is needed when no such chart exists, that is, when $\S^n\setminus U$ has empty interior; this is exactly the case $F(f) = U$ of Theorem \ref{thm:stable}, which contains the Cantor Julia set examples of \cite{HMM}.

The modulus metric $\mu_U$ is M\"obius invariant, since the conformal modulus is, and is monotone in the domain: if $B\subset U$ is a subdomain and $x,y\in B$, then $\mu_U(x,y) \leq \mu_B(x,y)$, because for a continuum $C\subset B$ containing $x,y$ every path in $\Delta(C,\partial U;U)$ has a subpath in $\Delta(C,\partial B;B)$, so $M(\Delta(C,\partial U;U)) \leq M(\Delta(C,\partial B;B))$ by minorization.

\begin{lemma}
\label{lem:muupper}
For every proper subdomain $U$ of $\S^n$ with $\operatorname{cap}\partial U >0$ we have $\mu_U \leq \hat{b}_1 \hat{k}_U + \hat{b}_2$, where $\hat{b}_1 = 4\Phi(\log 3)$ and $\hat{b}_2 = 2\Phi(\log 3)$ depend only on $n$.
\end{lemma}

\begin{proof}
We first show that $\mu_U(x,y) \leq b:=\Phi(\log 3)$ whenever $\hat{k}_U(x,y) \leq 1/4$. Rotations of $\S^n$ are M\"obius maps preserving $\sigma$, hence preserving $\mu_U$ and $\hat{k}_U$, so replacing $U$ by its image under a rotation we may assume $x = 0$. Put $\delta = \sigma(0,\partial U)$; then $0<\delta<\pi$, as $\delta = \pi$ would force $\partial U = \{\infty\}$, contradicting $\operatorname{cap}\partial U >0$. For $z$ in the spherical ball $\{ \sigma(0,z) < \delta/2 \}$ we have $\sigma(z,\partial U) \leq \sigma(z,0) + \delta \leq 3\delta/2$, so every path from $0$ leaving this ball has $\hat{k}_U$-length at least $(\delta/2)/(3\delta/2) = 1/3 > 1/4$. Hence $\sigma(0,y) < \delta/2$. Now the spherical ball $\{ \sigma(0,z) < t \}$, for $0<t<\pi$, is the Euclidean ball $B(0,r_t)$ with $r_t = \tan(t/2)$, and $B(0,r_{\delta}) \subset U$ by the definition of $\delta$. By monotonicity and conformal invariance of $\mu$,
\[ \mu_U(0,y) \leq \mu_{B(0,r_\delta)}(0,y) = \Phi ( \rho_{B(0,r_\delta)}(0,y)) = \Phi \left ( \log \frac{ r_\delta + |y| }{ r_\delta - |y| } \right ) \leq \Phi(\log 3),\]
since $|y| \leq r_{\delta/2} \leq r_\delta / 2$; the last inequality is $\tan(\delta/4)/\tan(\delta/2) = \tfrac12 (1-\tan^2(\delta/4)) \leq \tfrac12$.

For general $x,y$ put $t = \hat{k}_U(x,y)$ and choose a path from $x$ to $y$ of $\hat{k}_U$-length at most $t + 1/8$. Subdividing it into $m\leq 4t+2$ pieces of $\hat{k}_U$-length at most $1/4$ and using the triangle inequality for $\mu_U$ gives $\mu_U(x,y) \leq mb \leq 4bt+2b = \hat{b}_1 t + \hat{b}_2$.
\end{proof}

\begin{lemma}
\label{lem:transfer}
Let $U$ be a proper subdomain of $\S^n$ with $\partial U$ $\alpha$-uniformly perfect. Then $\S^n\setminus U$ is $\alpha'$-uniformly perfect with $\alpha' = \min\{ \alpha/3, 1/8\}$.
\end{lemma}

\begin{proof}
Write $E' = \S^n\setminus U$, a closed set containing $\partial U$, and let $x\in E'$ and $0<r<\operatorname{diam}_\sigma E'$; put $A = \{ y : \alpha' r \leq \sigma(x,y) \leq r\}$, which is non-empty as $r < \pi$. Let $t = \sigma(x,\partial U)$, attained at $x'\in\partial U$. We first note that $\overline{B}_\sigma(x,t) \subset E'$: if $y\in U$ with $\sigma(x,y)\leq t$, the last point $z$ of a spherical geodesic from $x$ to $y$ that lies in $E'$ belongs to $\partial U$ and satisfies $\sigma(x,z) < \sigma(x,y) \leq t$, a contradiction. 

If $r\leq t$, then $A \subset \overline{B}_\sigma(x,t) \subset E'$. If $\alpha' r \leq t < r$, then $x'\in A\cap E'$. So assume $t<\alpha' r$ and put $s = r-t \in ((1-\alpha')r, r)$.

If $s < \operatorname{diam}_\sigma \partial U$, uniform perfectness of $\partial U$ at $x'$ gives $y\in\partial U$ with $\alpha s \leq \sigma(x',y) \leq s$; then $\sigma(x,y) \leq t+s = r$ and $\sigma(x,y) \geq \alpha s - t \geq \alpha(1-\alpha')r - \alpha' r \geq (\alpha - 2\alpha')r \geq \alpha' r$, so $y\in A\cap E'$.

If $s \geq \operatorname{diam}_\sigma\partial U =: D$, then $\partial U\subset \overline{B}_\sigma(x',D)$, and the connected set $\S^n\setminus \overline{B}_\sigma(x',D)$ (recall $n\geq 2$ and $D\leq s<\pi$) is disjoint from $\partial U$, hence contained in $U$ or in $E'$. In the latter case $E'$, being closed, contains $\S^n\setminus B_\sigma(x',D)$; pick $y$ with $\sigma(x,y) = r$ and $\sigma(x',y)\geq D$, which exists since the sphere $\{ y : \sigma(x,y) = r\}$ contains points at distance $\min\{ r+t , 2\pi - r - t\} \geq r-t = s \geq D$ from $x'$; then $y\in A\cap E'$. In the former case $E'\subset \overline{B}_\sigma(x',D)$, so $r<\operatorname{diam}_\sigma E' \leq 2D$; choose $y\in\partial U$ with $\sigma(x',y)\geq D/2$. Then $\sigma(x,y) \leq D+t \leq r$ and $\sigma(x,y) \geq D/2 - t > r/4 - \alpha' r \geq \alpha' r$, as $\alpha'\leq 1/8$. So $y\in A\cap E'$ in this case too.
\end{proof}

\begin{lemma}
\label{lem:mucomp}
If $U$ is uniform with uniformly perfect boundary, then $\operatorname{cap}\partial U>0$, so $\mu_U$ is a metric, and there is a constant $\hat{c}_1>0$, depending only on $n$, the uniformity constant of $U$, the uniform perfectness constant of $\partial U$ and the spherical diameter of $\partial U$, such that
\[ \hat{c}_1 \hat{k}_U \leq \mu_U \leq \hat{b}_1 \hat{k}_U +\hat{b}_2.\]
\end{lemma}

\begin{proof}
Uniformly perfect sets have positive capacity by \cite[Theorem 4.1 and Corollary 4.2]{JV96}, and the upper bound is Lemma \ref{lem:muupper}. For the lower bound, by Lemma \ref{lem:transfer} the closed set $E = \S^n\setminus U$ is uniformly perfect, with constant depending only on that of $\partial U$, and $\operatorname{diam}_\sigma E \geq \operatorname{diam}_\sigma \partial U$. By \cite[Theorem 4.4]{JV96}, $E$ satisfies Vuorinen's metric thickness condition \cite[(2.21)]{Vuo87} with parameters $r_0 = \tfrac14 q(E) \geq \tfrac14 q(\partial U)$ and $\delta = \delta(n,\alpha')$. Hence \cite[Theorem 2.27]{Vuo87} gives $\mu_U \geq a_1 j_U$ with $a_1 = a_1(n,r_0,\delta)$, where $j_U$ is defined as $\hat{\jmath}_U$ but with the chordal metric; since $q\leq\sigma\leq\tfrac{\pi}{2}q$ and $\log(1+\lambda u) \leq \lambda\log(1+u)$ for $\lambda\geq 1$, see \cite[(2.12)]{BHK}, we have $\hat{\jmath}_U \leq \tfrac{\pi}{2} j_U$. Finally, in an $A$-uniform metric space the Gehring--Osgood inequality $\hat{k}_U \leq 4A \hat{\jmath}_U$ holds, see \cite[Lemma 2.13, (2.16)]{BHK}. Hence $\hat{c}_1 = a_1/(2\pi A)$ works.
\end{proof}

\begin{lemma}
\label{lem:unifgromov}
Let $U$ be uniform. Then $(U,\hat{k}_U)$ is a proper geodesic Gromov hyperbolic space, and every Gromov sequence $(z_k)$ for $\hat{k}_U$ converges in $\S^n$ to a point of $\partial U$, equivalent Gromov sequences having the same limit. Moreover, $\hat{k}_U(x,y) \leq M$ implies $\sigma(x,y) \leq (e^M-1) \sigma(x,\partial U)$.
\end{lemma}

\begin{proof}
The last assertion is the elementary inequality $\log ( 1 + \sigma(x,y)/\min\{\sigma(x,\partial U),\sigma(y,\partial U)\}) \leq \hat{k}_U(x,y)$ of \cite[(2.4)]{BHK}.

That $(U,\hat{k}_U)$ is proper, geodesic and Gromov hyperbolic is the first assertion of \cite[Theorem 3.6]{BHK} applied to the uniform space $(U,\sigma)$, which requires no boundedness; see \cite[Chapter 7]{BHK} for domains in $\S^n$. As $(U,\sigma)$ is bounded, \cite[Proposition 3.12 (b),(c),(d)]{BHK} shows that every $\hat{k}_U$-geodesic ray has finite $\sigma$-length and a unique endpoint in $\partial U$, and that ray $\mapsto$ endpoint is a bijection from the Gromov boundary onto $\partial U$. Given a Gromov sequence $(z_k)$ with base point $o$, it defines a point $\xi$ of the Gromov boundary; let $\gamma$ be a geodesic ray from $o$ to $\xi$ with endpoint $w\in\partial U$. The geodesic segments $[o,z_k]$ subconverge to a ray equivalent to $\gamma$, so for large $k$ there are $s_k \to \infty$ with $\hat{k}_U(z_k,\gamma(s_k)) \leq M$, where $M$ depends only on the hyperbolicity constant. Since $\gamma(s) \to w$ we have $\sigma(\gamma(s_k),\partial U) \to 0$, and the estimate of the first paragraph gives $\sigma(z_k,\gamma(s_k)) \to 0$, so $z_k \to w$. Two equivalent Gromov sequences define the same $\xi$, hence the same $w$.
\end{proof}

\begin{proof}[Proof of Theorem \ref{thm:stable}]
Let $U$ be a stable component of $F(f)$ which is uniform with uniformly perfect boundary. We first arrange that $U\subset\R^n$: since $f$ has degree at least $2$, $J(f) \neq \emptyset$, and conjugating $f$ by a M\"obius map sending a point of $J(f)$ to $\infty$ replaces $f$ by a uniformly quasiregular map, $U$ by a stable component of its Fatou set contained in $\R^n$, and preserves all the hypotheses and conclusions: $\mu$ is M\"obius invariant, uniformity is invariant under M\"obius maps by \cite[Theorem 2.10]{Vai88}, uniform perfectness with respect to $\sigma$ is preserved because M\"obius maps of $\S^n$ are bi-Lipschitz for $\sigma$ (with constant depending on the map), and the dynamics is conjugated. So we assume $U\subset\R^n$, possibly unbounded, with $\infty\in\partial U$ allowed.

Let $G$ be the invariant structure of $f$ restricted to $U$; then $f|_U:U\to U$ is uniformly quasiregular with $f^*G = G$ on $U$, and $d_G$ is a metric on $U$ by Proposition \ref{prop:dGmetric}, using Lemma \ref{lem:mucomp}. Theorem \ref{thm:sp} holds, and by \eqref{eq:mucomps} and Lemma \ref{lem:mucomp},
\begin{equation}
\label{eq:dGkhat}
\lambda^{-n} \hat{c}_1 \hat{k}_U \leq d_G \leq \lambda^n (\hat{b}_1 \hat{k}_U + \hat{b}_2).
\end{equation}
This replaces Lemma \ref{lem:roughhyp} in establishing that $d_G$ induces the topology of $U$, that $d_G$-bounded sets are exactly the relatively compact subsets of $U$, and that $(U,d_G)$ is proper, since $\hat{k}_U$ has these properties. Define chain metrics on $U$ as in Definition \ref{def:chain} with $r \geq \max\{ r_1 , \lambda^n(\hat{b}_1+\hat{b}_2)\}$, so that steps of $\hat{k}_U$-length at most $1$ are admissible. Then Lemma \ref{lem:chainprop} and Proposition \ref{prop:roughgeos} hold with $\hat{k}_U$ in place of $\rho$ and \eqref{eq:dGkhat} in place of Lemma \ref{lem:roughhyp}, so the identity is a quasi-isometry $(U,d) \to (U,\hat{k}_U)$, and Theorem \ref{thm:gromovchain} holds with $Y = (U,\hat{k}_U)$ by Lemma \ref{lem:unifgromov}, with $\delta_0$ the hyperbolicity constant of $(U,\hat{k}_U)$ from \cite[Theorem 3.6]{BHK}: its proof used of $Y$ only that it is proper, geodesic and Gromov hyperbolic and that its Gromov sequences converge in the ambient space to boundary points, equivalent sequences to the same point.

The arguments of Section 7 now apply to $f|_U$ with $d$ the chain metric on $U$, with one change: in Proposition \ref{prop:escape}~(a),(b),(d), the fact that a $\rho$-ball of fixed radius about points tending to $\S^{n-1}$ has Euclidean diameter tending to $0$ is replaced by the last assertion of Lemma \ref{lem:unifgromov} together with $d \geq \tfrac12 \lambda^{-n}\hat{c}_1 \hat{k}_U$, Euclidean convergence by convergence in $\S^n$, and $\S^{n-1}$ by $\partial U$ in Proposition \ref{prop:escape}~(c). In the bounded case, \cite[Propositions 4.6 and 4.9]{HMM} show that $U$ is a (super-)attracting basin or a rotation domain, exactly as in Proposition \ref{prop:hmm}. In the escaping case, the proof of Theorem \ref{thm:dw} gives $w\in\partial U$ with $f^m\to w$ locally uniformly on $U$; since $f$ is continuous on $\S^n$, $f(w) = \lim_m f(f^m(x)) = w$, so $w$ is a fixed point of $f$ and $U$ is a parabolic basin, and $w$ is unique because the full sequence converges.

Finally, the listed classes satisfy the hypotheses. Quasiballs are uniform domains, see \cite[Theorem 2.15]{MS79}, and, in common with all uniform domains with connected boundary, have uniformly perfect boundary by \cite[Example 5.1(1)]{JV96}; so do uniform domains whose boundary has finitely many non-degenerate components; and if $\partial U = J(f)$ then $\partial U$ is uniformly perfect by \cite{FN11}.
\end{proof}

\begin{proof}[Proof of Theorem \ref{thm:coarse}]
This is the case $G\equiv I$ of the argument just given, with $\lambda = 1$ and $d_G = \mu_U$, and no dynamics: for $r\geq \max\{r_1,\hat{b}_1+\hat{b}_2\}$ and $\psi\in\Psi$, Lemma \ref{lem:chainprop} with $\hat{k}_U$ in place of $\rho$ shows that $(U,d_r^{\psi})$ is quasi-isometric via the identity to $(U,\mu_U)$ and to $(U,\hat{k}_U)$, Proposition \ref{prop:roughgeos} shows that it is roughly geodesic, and Theorem \ref{thm:gromovchain} with $Y = (U,\hat{k}_U)$, which is legitimate by Lemma \ref{lem:unifgromov}, shows that it is Gromov hyperbolic and that its Gromov sequences and their equivalence agree with those of $\hat{k}_U$. Since the Gromov boundary of $(U,\hat{k}_U)$ is $\partial U$ by \cite[Proposition 3.12]{BHK}, so is that of $(U,d_r^{\psi})$.
\end{proof}

We remark that if $U$ is a quasiball, Lemmas \ref{lem:muupper}--\ref{lem:unifgromov} can be avoided: $F = \varphi^{-1}\circ f\circ\varphi$ is uniformly quasiregular on $\B^n$ for a quasiconformal map $\varphi:\S^n\to\S^n$ with $\varphi(\B^n) = U$, Theorem \ref{thm:dw} applies to $F$, and $\varphi$ extends to a homeomorphism of $\overline{\B^n}$ onto $\overline{U}$.

A bounded domain is uniform if and only if its quasihyperbolic metric is Gromov hyperbolic and the natural map identifies the Gromov boundary with the Euclidean boundary quasisymmetrically \cite[Theorem 1.11]{BHK}, and the same holds for domains in $\S^n$ with the spherical metric \cite[Theorems 3.6 and 7.11]{BHK}. Therefore the ``uniform" condition is the natural one for Theorem \ref{thm:gromovchain}. The ``uniformly perfect" condition is needed for the lower bound in Lemma \ref{lem:mucomp} and cannot be omitted: the punctured ball is a uniform domain, but $\mu$ does not see the puncture, while $\hat{k}$ does, see \cite[p. 91]{Vuo85} and \cite{Vuo87}.

\section{Examples}

In this section, we prove Theorem \ref{thm:examples}. We first address the hyperbolic case. The idea is to construct a conformal trap in the vein of Iwaniec and Martin \cite{IM96}.

\begin{lemma}
\label{lem:strip}
Let $h>0$. There exists $m_0 = m_0(h)$ such that for every integer $m\geq m_0$ the following holds. Put $a=m^{-1/2}$, $y_3 = \log m/(m-1)$, $y_c = \log m/(2(m-1))$, and let $\beta \in C^{\infty}(\R)$ satisfy $0\leq \beta \leq 1$, $\beta = 1$ on $(-\infty, y_3]$, $\beta = 0$ on $[h/2,\infty)$, and $|\beta'| \leq 4/h$. On the strip $T = \{ z \in \C : 0<\Im z < h \}$ define
\[ w(z) = z - i \operatorname{Log} ( 1 + a\beta ( \Im z ) e^{i(m-1)z} ).\]
Then:
\begin{enumerate}[(i)]
\item $w$ is $C^{\infty}$ on $T$, $w(z+2\pi) = w(z) + 2\pi$, and $w(z) = z$ for $\Im z \geq h/2$;
\item $\Im w(z) \geq \Im z - \log (1+a) > -a$;
\item $w$ is $K_m$-quasiregular on $T$ with $K_m \to 1$ as $m\to\infty$;
\item $w$ is holomorphic on $\{ \Im z \leq y_3 \}$ and has critical points there, all on the line $\Im z = y_c$, each of local index $2$; in particular $w$ is not injective.
\end{enumerate}
Let $Z_c = \{ z_c \}$ be the set of critical points of $w$, put 
\[ \epsilon := \tfrac12 \min \{ y_c , \pi/(m-1) \} ,\]
and define $\psi : T \to T$ by
\[ \psi ( z_c + re^{i\vartheta} ) = z_c + \sqrt{\epsilon r}\, e^{i\vartheta} \quad (z_c \in Z_c,\ 0\leq r \leq \epsilon), \qquad \psi = \operatorname{id} \text{ elsewhere},\]
and $W := w \circ \psi$. Then, after decreasing $\epsilon$ if necessary (depending on $h$ and $m$ only):
\begin{enumerate}[(i)]
\setcounter{enumi}{4}
\item $\psi$ is a $2$-quasiconformal homeomorphism of $T$ onto itself commuting with $z\mapsto z+2\pi$, and $W$ is $2K_m$-quasiregular on $T$ with $W = w$ off the disks $D(z_c,\epsilon)$; in particular (i) and (ii) hold with $W$ in place of $w$;
\item there are constants $L$ and $c_0>0$, depending only on $h$ and $m$, such that $|DW| \leq L$ and $J_W \geq c_0$ almost everywhere on $T$;
\item the branch set of $W$ is $Z_c$, each $z_c$ has local index $2$, and near $z_c$ the map $W$ is a winding map: $W(z_c + re^{i\vartheta}) = w(z_c) + \epsilon r e^{2i\vartheta} g_c ( \psi (z_c + re^{i\vartheta}))$ with $g_c$ holomorphic and non-vanishing on $D(z_c,\epsilon)$. In particular $W$ is not injective.
\end{enumerate}
\end{lemma}

\begin{proof}
Write $y = \Im z$, $E(z) = e^{i(m-1)z}$, $u = a\beta(y)E$. Since $|E| = e^{-(m-1)y} \leq 1$ on $T$, we have $|u| \leq a <1$ and so $\operatorname{Log}(1+u)$ is the principal branch and $w$ is $C^{\infty}$. Periodicity and~(i) are clear. For~(ii), we have
\[ \Im w = y - \Re \operatorname{Log} (1+u) \geq y - \log (1+|u|).\]
For~(iv), where $\beta \equiv 1$, $w$ is holomorphic with
\[ w'(z) = 1+ \frac{ (m-1)aE}{1+aE} = \frac{ 1+maE}{1+aE},\]
which vanishes if and only if $maE = -1$, that is, $e^{-(m-1)y} = (ma)^{-1} = m^{-1/2}$ and $(m-1) \Re z \in \pi + 2\pi \Z$. So the critical points lie on $y = \log (m^{1/2}) / (m-1) = y_c < y_3$, and they are simple zeros of $w'$ since $E' \neq 0$. This gives local index $2$. We need $y_3 < h/2$, which holds for $m\geq m_0(h)$. Finally, for~(iii), on $\{ y \leq y_3 \}$, $w$ is holomorphic and on $\{y \geq h/2 \}$ it is the identity. On the strip $\{ y_3 \leq y \leq h/2 \}$, using $y_{\overline{z}} = i/2$, $y_z = -i/2$, $E_{\overline{z}} = 0$, $E_z = i(m-1)E$, we have
\[ w_{\overline{z}} = \frac{ -i u_{\overline{z}} }{1+u} = \frac{ a\beta'(y) E/2 }{1+u},\quad w_z = 1 - \frac{ -iu_z}{1+u} = 1 - \frac{a\beta'(y) E/2}{1+u} + \frac{ (m-1)a\beta(y)E}{1+u}.\]
On the strip we have $|E| \leq e^{-(m-1)y_3} = m^{-1}$, so with $|\beta'|\leq 4/h$ and $|1+u| \geq 1-a$, 
\[ | w_{\overline{z}}| \leq \frac{ 2m^{-3/2}}{h(1-m^{-1/2})} ,\quad |w_z| \geq 1 - \frac{ 2m^{-3/2}}{h(1-m^{-1/2})} - \frac{ m^{-1/2}}{1-m^{-1/2}}.\]
For $m\geq m_0(h)$, we have $|w_z| \geq 1/2$, so $J_w = |w_z|^2 - |w_{\overline{z}}|^2 >0$ and the maximal dilatation $K(w)$ is at most $K_m \to 1$. Off the strip we have $K=1$.

For~(v)--(vii), first note that $Z_c$ is $2\pi$-periodic, that consecutive points of $Z_c$ are at distance $2\pi/(m-1)$, and that $Z_c \subset \{ \Im z = y_c \}$ with $y_c < y_3 = 2y_c$; so by the choice of $\epsilon$ the disks $D(z_c,\epsilon)$ are pairwise disjoint and contained in $\{ y_c/2 < \Im z < y_3 \}$, where $w$ is holomorphic. 

For~(v), on $D(z_c,\epsilon)$ the map $\psi$ is a radial homeomorphism with radial derivative $\tfrac12 \sqrt{\epsilon/r}$ and tangential stretch $\sqrt{\epsilon r}/r = \sqrt{\epsilon/r}$, so $\psi$ is sense-preserving with linear dilatation $2$ and $J_{\psi} = \epsilon/(2r)$; it is the identity on $\partial D(z_c,\epsilon)$, hence continuous on $T$, and the identity off the disks. Thus $\psi$ is a $2$-quasiconformal homeomorphism of $T$ onto itself, and $W = w\circ \psi$ is $2K_m$-quasiregular. Periodicity is clear, $W = w$ off the disks gives $W(z) = z$ for $\Im z\geq h/2$, and since $\psi(T)= T$, (ii) for $w$ gives $\Im W(z) = \Im w(\psi(z)) \geq \Im \psi(z) - \log (1+a) > -a$.

For~(vi) and~(vii), fix $z_c\in Z_c$. Since $z_c$ is a simple zero of $w'$ and $w$ is holomorphic near $z_c$, we can write $w(z) = w(z_c) + (z-z_c)^2 g_c(z)$ with $g_c$ holomorphic on a neighbourhood of $\overline{D(z_c,\epsilon)}$ and $g_c(z_c) = w''(z_c)/2 \neq 0$. By periodicity there are only finitely many $z_c$ modulo $2\pi$, so after decreasing $\epsilon$ we may assume that $|g_c| \geq g_-$, $|g_c| \leq g_+$ and $|g_c'| \leq g_+'$ on $D(z_c,\epsilon)$ for all $z_c$, with $g_-,g_+,g_+'>0$ depending on $h,m$ only, and that $\epsilon g_+' < g_-$. Writing $\zeta = \psi(z)$ for $z = z_c + re^{i\vartheta}$, so that $\zeta - z_c = \sqrt{\epsilon r}\, e^{i\vartheta}$, we get the formula in~(vii), and
\[ w'(\zeta) = (\zeta - z_c) \left ( 2 g_c(\zeta) + (\zeta - z_c) g_c'(\zeta) \right ), \qquad g_- |\zeta - z_c| \leq |w'(\zeta)| \leq (2g_+ + \epsilon g_+')|\zeta - z_c| .\] Since $DW(z) = Dw(\psi(z)) D\psi(z)$ and $|D\psi(z)| = \sqrt{\epsilon/r}$, 
\[ |DW(z)| \leq (2g_++\epsilon g_+') \sqrt{\epsilon r}\, \sqrt{\epsilon /r} = (2g_++\epsilon g_+')\epsilon, \qquad J_W(z) = |w'(\zeta)|^2 J_{\psi}(z) \geq g_-^2 \epsilon r \cdot \frac{\epsilon}{2r} = \frac{g_-^2\epsilon^2}{2}\]
on $D(z_c,\epsilon)\setminus \{z_c\}$. Off the disks, $W = w$. On $\{ \Im z \geq h/2 \}$, $w$ is the identity. On the band $\{ y_3 \leq \Im z \leq h/2\}$ the estimates in the proof of~(iii) give $|w_z| \leq 1 + \tfrac{2m^{-3/2}}{h(1-m^{-1/2})} + \tfrac{m^{-1/2}}{1-m^{-1/2}}$, $|w_{\overline{z}}| \leq \tfrac{2m^{-3/2}}{h(1-m^{-1/2})}$ and $|w_z|\geq 1/2$, so $|Dw| = |w_z| + |w_{\overline z}|$ is bounded and $J_w = |w_z|^2 - |w_{\overline{z}}|^2 \geq 1/8$ for $m\geq m_0(h)$. On $\{ \Im z \leq y_3\}$ minus the disks, $w$ is holomorphic with 
\[ w'(z) = \frac{1+maE(z)}{1+aE(z)}, \qquad E(z) = e^{i(m-1)z}, \quad |E(z)| = e^{-(m-1)\Im z} \leq 1,\]
so $|w'| \leq (1+ma)/(1-a)$ is bounded, while $1+maE$ vanishes exactly on $Z_c$. Put $\delta = \epsilon/2$. For $\Im z \leq y_c - \delta$ we have $ma|E(z)| \geq ma\, e^{-(m-1)(y_c-\delta)} = e^{(m-1)\delta}$, using $ma\,e^{-(m-1)y_c} = 1$, so $|1+maE| \geq e^{(m-1)\delta} - 1$; for $\Im z \geq y_c + \delta$ we have $|1+maE| \geq 1 - e^{-(m-1)\delta}$; and on the set $\{ y_c - \delta \leq \Im z \leq y_c + \delta,\ 0\leq \Re z \leq 2\pi \} \setminus \bigcup_{z_c} D(z_c,\epsilon)$, which is compact and contains no zero of $1+maE$, the continuous function $|1+maE|$ has a positive minimum. Hence $|w'|$, and so $J_w = |w'|^2$, is bounded below off the disks, and~(vi) follows with $L$ and $c_0$ the worst of the constants above.

For~(vii), $W$ is a local homeomorphism wherever it is $C^1$ with $J_W>0$, which by the above is everywhere except at the points $z_c$; near $z_c$ the formula shows that $W - w(z_c)$ is the winding map $re^{i\vartheta} \mapsto \epsilon r e^{2i\vartheta}$ multiplied by the non-vanishing continuous factor $g_c\circ\psi$, so $W$ has local index $2$ at $z_c$ and is $2$-to-$1$ on $D(z_c,\epsilon)\setminus\{z_c\}$.
\end{proof}

We remark on why the passage from $w$ to $W$ is necessary. The product of a holomorphic map having a critical point with the identity in the remaining coordinates is {\it not} quasiregular: near the critical set its Jacobian tends to $0$ while its operator norm stays bounded below by $1$, so $|Df|^n \leq K J_f$ fails. The radial correction $\psi$ converts the critical points of $w$ into winding-type branch points, at which the Jacobian is bounded below, and it is this bounded length distortion property of $W$ in the sense of Martio and V\"ais\"al\"a \cite{MV88} that makes the product $W\times \operatorname{id}$ quasiregular in Proposition \ref{prop:hypex}. This is also why winding maps, rather than powers, are the basic branched quasiregular maps in higher dimensions.

To build our example, we will work in the upper half-space $\Hp^n = \{x_n > 0\}$ with coordinates $x = (x_1,x'',x_n)$, $x''\in \R^{n-2}$ (absent when $n=2$). A M\"obius transformation transports everything to $\B^n$ and all the important properties are M\"obius invariant.

\begin{proposition}
\label{prop:hypex}
Let $n\geq 2$ and $\theta >1$. Choose $s \in (0,1-\frac{1}{\theta} )$, put $h = 1 - \frac{1}{\theta} - s$, and let $m \geq m_0(h)$ be an integer with $m^{-1/2} < s$. Let $W$ be the map from Lemma \ref{lem:strip} for these $h,m$. Define, on $S = \{0<x_n<1 \} \subset \Hp^n$,
\[ A(x) = \left ( x_1 , x'' , \frac{1}{\theta} +s+hx_n \right ), \quad \Omega(x) = \left ( \Re W(z) , x'' , \Im W(z) + \frac{1}{\theta} + s \right ), \]
with $z = x_1 + i(x_n - \frac{1}{\theta} -s )$, and
\[ f(x) = \left \{ \begin {array}{cc} \theta x, & x_n \geq 1,\\ \theta \Omega(A(x)), & 0<x_n<1 . \end{array} \right . \]
Then $f:\Hp^n \to \Hp^n$ is a non-injective uniformly quasiregular map, $f$ has no fixed point in $\Hp^n$, and $f^k \to \infty$ locally uniformly on $\Hp^n$. In the ball model, $f$ is the escaping case of Theorem \ref{thm:dw} with Denjoy-Wolff point the image of $\infty$, the hyperbolic M\"obius map $x\mapsto\theta x$ as conformal model, and of positive-step type: $s_k \geq \tfrac12 \lambda^{-n} 2^{n-1}c_n \log \theta > 0$ for all $k\geq 1$, hence for all $k$ by monotonicity.
\end{proposition}

\begin{proof}
Write $V = \{x_n>1\}$, $g(x) = \theta x$, and $T' = A(S) = \{ \frac{1}{\theta} +s < x_n <1 \}$. For $x\in T'$, $z$ lies in the strip $T$ of Lemma \ref{lem:strip} since $0<x_n - \frac{1}{\theta} -s < h$.

As $A$ is the identity on $\{x_n=1\}$, and as for $x\in T'$ with $x_n \geq \frac{1}{\theta} + s +h/2$ we have $\Im z \geq h/2$, it follows that $W(z)=z$ and $\Omega(x)=x$. Hence on $\{\tfrac12 \leq x_n < 1 \}$ we have $f = \theta A$, which equals $\theta x$ at $x_n=1$. We conclude that $f$ is continuous on $\Hp^n$.

For $x\in S$, Lemma \ref{lem:strip}~(ii) gives $\Omega(A(x))_n = \Im W(z) + \frac{1}{\theta} +s > \frac{1}{\theta} +s-a > \frac{1}{\theta}$, since $a=m^{-1/2} < s$. So $f(S) \subset \{x_n >1 \} = V$ and $f(\overline{V}) = \{x_n \geq \theta \} \subset V$. In particular, $f(\Hp^n) \subset V \subset \Hp ^n$.

To show that $f$ is quasiregular, consider first $\Omega$ on $T'$. Writing $z = x_1 + i(x_n - \frac{1}{\theta} - s)$ and $W = W_1 + iW_2$, we have $\Omega(x) = (W_1 , x'', W_2 + \frac{1}{\theta}+s)$, so $D\Omega$ is, after permuting coordinates, the block matrix $\operatorname{diag}(DW, I_{n-2})$. Hence $|D\Omega| = \max \{ |DW|, 1\}$ and $J_{\Omega} = J_W$, and by Lemma \ref{lem:strip}~(vi),
\[ |D\Omega|^n \leq \max\{L,1\}^n \leq \frac{ \max\{L,1\}^n }{c_0} J_{\Omega} \quad \text{a.e. on } T'.\]
Moreover $\Omega$ is locally Lipschitz (since $W$ is), hence in $W^{1,n}_{\rm loc}(T')$, and $J_\Omega \geq c_0 >0$ a.e. Therefore $\Omega$ is $K_{\Omega}$-quasiregular on $T'$ with $K_{\Omega} = \max\{L,1\}^n/c_0$. This is the only place where Lemma \ref{lem:strip}~(vi) is used; see the remark following that lemma for why it is needed. As $A$ is an affine bijection of $S$ onto $T'$ with constant dilatation and $x\mapsto \theta x$ is conformal, $f|_S = \theta\, \Omega\circ A$ is quasiregular on $S$ with dilatation bounded in terms of $K_\Omega$ and $K(A)$.

Next we glue across the hyperplane $\Pi = \{ x_n = 1\}$. On $\{ \tfrac12 \leq x_n < 1 \}$ we have $f = \theta A$, and on $\{ x_n \geq 1\}$ we have $f = \theta\,\operatorname{id}$; these are affine maps agreeing on $\Pi$, so $f$ is locally Lipschitz on the open set $\{ \tfrac12 < x_n < \tfrac32 \}$. Hence $f \in W^{1,n}_{\rm loc}(\Hp^n)$, and since $\Pi$ has measure zero the inequality $|Df|^n \leq K J_f$ holds a.e. on $\Hp^n$ with $K$ the larger of the dilatations on $S$ and on $V$ (the latter being $1$). Thus $f:\Hp^n \to \Hp^n$ is $K$-quasiregular. 

That $f$ is non-injective follows from Lemma \ref{lem:strip}~(vii): $W$ is $2$-to-$1$ near each $z_c$, and these points correspond to points of $S$, namely those $x$ with $A(x)_n - \frac{1}{\theta} - s = y_c$, i.e. $x_n = y_c/h \in (0,1)$; as $A$ and $x\mapsto\theta x$ are injective, $f$ is $2$-to-$1$ near such points. The branch set of $f$ contains $A^{-1}(\{ z_c\} \times \R^{n-2})$, a union of $(n-2)$-planes, and the local behaviour of $f$ there is that of a winding map about an $(n-2)$-plane.

Since $f(\Hp^n) \subset \overline{V}$ and $f = g$ on $\overline{V}$ with $g(\overline{V}) \subset V$, we have $f^k = g^{k-1}\circ f$ for every $k\geq 1$; as $g$ is conformal, $K(f^k) = K(f)$ and $f$ is uniformly quasiregular. Every orbit passes through the non-conformal region $S$ at most once, at time $0$, which is the whole content of the conformal trap. Consequently $f^k(x) = \theta^{k-1} f(x)$ for all $x$, and since $f$ maps compact subsets of $\Hp^n$ to compact subsets of $\Hp^n$, $f^k \to \infty$ locally uniformly on $\Hp^n$. There are no fixed points: on $V$, $\theta x = x$ forces $x = 0 \notin \Hp^n$, and on $S$ we have $f(x) \in V$ while $x\notin V$.

Finally we compute the step type. Let $G$ be an invariant conformal structure for $f$, $\lambda = H^{1/2}$ the constant of \eqref{eq:vcomp}, and $d$ the chain metric, transported to $\Hp^n$ by a M\"obius map; Lemma \ref{lem:chainprop}~(c) gives $d \geq \tfrac12 \lambda^{-n} 2^{n-1} c_n \rho$, with $\rho$ the hyperbolic metric of $\Hp^n$. For $k\geq 1$ we have $x_{k+1} = \theta x_k$, and for every $y\in\Hp^n$, $\rho(y,\theta y) \geq \log \theta$, since $\log\theta$ is the translation length of the hyperbolic M\"obius transformation $y\mapsto \theta y$ (equality holds exactly on its axis $\{ y' = 0\}$; in general $\cosh \rho(y,\theta y) = 1 + (\theta-1)^2|y|^2/(2\theta y_n^2) \geq 1 + (\theta-1)^2/(2\theta) = \cosh \log\theta$). Hence
\[ s_k = d(x_k , \theta x_k) \geq \tfrac12 \lambda^{-n} 2^{n-1} c_n \log \theta > 0 \quad (k\geq 1),\]
so $f$ is of positive-step type. In the ball model, $f$ is a non-injective uniformly quasiregular self-map in the escaping case of Theorem \ref{thm:dw}, with Denjoy-Wolff point the image of $\infty$, and its conformal model $g$ at that point is a hyperbolic M\"obius transformation.
\end{proof}

For parabolic examples, \cite[Theorem 5.2]{HMM} produces, for every $n\geq 2$, a uniformly quasiregular map $g:\S^n \to \S^n$ with a parabolic fixed point $x_1$, whose Fatou set is a single completely invariant parabolic basin and whose Julia set is a Cantor set. The relevant features of the construction are:
\begin{itemize}
\item there is a ball $B$ with $x_1 \in \partial B$ on which $g = \Phi \circ R$ is a parabolic M\"obius transformation with fixed point $x_1$ and $g(B) \subset B$. Here, $R$ is a rotation by $\pi$ about an $(n-2)$-plane through $x_1$ and $\Phi$ is a M\"obius map exchanging $B$ with its complement, so that $B_1 := R(B)$ is tangent to $B$ at $x_1$ from the other side of the common tangent hyperplane;
\item there are pairwise disjoint balls $B_2,\ldots, B_d$, each contained in a small neighbourhood of a point $x_j$, for $j\geq 2$, such that $g_1^{-1}(B) = B_1 \cup \ldots \cup B_d$, for the auxiliary map $g_1$, and $g = \Phi \circ g_1$. Consequently, every $x \notin B_1 \cup \ldots \cup B_d$ satisfies $g(x) \in B$;
\item the map $g$ is obtained from an arbitrary non-injective quasiregular map $f$ by modifications supported in balls of arbitrarily small radius around $x_1,\ldots, x_d$, and the radii of $B_2,\ldots, B_d$ may be taken arbitrarily small.
\end{itemize}

\begin{proposition}
\label{prop:paraex}
Let $g$ be as above. Let $B'$ be any ball with $B\subset B'$, $x_1 \in \partial B'$, and $x_2,\ldots, x_d \notin \overline{B'}$, and let the balls $B_2,\ldots, B_d$ in the construction, which \cite[Theorem 5.2]{HMM} allows to be taken arbitrarily small, be chosen small enough to be disjoint from $B'$. Then $g(B') \subset B'$, $g|_{B'}$ is uniformly quasiregular, $g$ has no fixed point in $B'$, and $g^k \to x_1$ locally uniformly on $B'$ with step limit $s>0$ and the parabolic M\"obius map $\Phi\circ R$ as conformal model. If in addition $B'$ contains two points identified by the original map $f$ at which $f$ was not modified, then $g|_{B'}$ is non-injective, such choices being possible.
\end{proposition}

\begin{proof}
$B'$ lies on the same side of the tangent hyperplane at $x_1$ as $B$, so $B' \cap B_1 = \emptyset$, and $B' \cap B_j = \emptyset$ for $j\geq 2$ by choice. Hence for $x\in B' \setminus B$, $g(x) \in B \subset B'$, and $g(B) \subset B$ so $g(B') \subset B'$. Restrictions of uniformly quasiregular maps to invariant domains are uniformly quasiregular. A fixed point in $B'$ would lie in $B$ as points of $B'\setminus B$ move into $B$, where $g$ is a parabolic M\"obius map without fixed points.

For $x\in B'$, $g(x) \in B$ and $g^k(x) = (\Phi \circ R)^{k-1}(g(x)) \to x_1$ locally uniformly since $g$ maps compact subsets of $B'$ to compact subsets of $B$ and the parabolic map converges locally uniformly on $B$ to its fixed point. The step limit is positive: in the model in which $B'$ is the upper half-space and $x_1 = \infty$, the ball $B$ is a half-space $\{y_n>c \}$ with $c>0$ and $\Phi \circ R$ is a parabolic Euclidean isometry $y\mapsto Oy' + v'$, with $O$ an orthogonal map on the $(n-1)$-dimensional space $(y_1,\ldots, y_{n-1})$ and leaving $y_n$ preserved. This isometry has no fixed point in $\R^{n-1}$, so
\[ |(\Phi\circ R)(y) - y| = |(O-I)y'+v'| \geq c_0 := \dist (v' , \operatorname{im}(O-I) ) >0 \]
for all $y$. Along an orbit $y_n$ is constant, so 
\[ \rho_{B'} ( g^k(x) , g^{k+1}(x) ) \geq \cosh^{-1} ( 1+ c_0^2 / 2y_n^2 ) >0 \]
is bounded below, and $s>0$ by Lemma \ref{lem:chainprop}~(c).

At points where $f$ is unmodified, $g =\Phi \circ R \circ f$ with $\Phi,R$ injective, so two points of $B'$ identified by $f$ are identified by $g$. Finally, to see the choices are possible, start from $f$ a winding map of degree $2$ about an $(n-2)$-plane $P$, whose fibres are pairs $\{y,R_P(y) \}$ with $R_P$ the rotation by $\pi$ about $P$. Take $x_1$ at distance $\ell$ from $P$, so $x_2 = R_P(x_1)$ is at distance $2\ell$ from $x_1$, and a ball $B'$ of radius slightly less than $\ell$ containing $B$ and tangent to it at $x_1$, positioned so as to contain a pair $\{ y,R_P(y) \}$ with $y$ close to $P$. The modifications in $x_1,x_2$ are made in balls small enough not to contain $y$ or $R_P(y)$.
\end{proof}

\section{Open questions}

We finish with some questions raised by the results above.

\begin{enumerate}[1.]
\item Theorem \ref{thm:sp}~(c) shows that a non-injective map preserving $G$ contracts $d_G$ strictly somewhere. Is $d_G(f(x),f(y)) < d_G(x,y)$ for all $x\neq y$? In contrast to the classical Schwarz-Pick Lemma, here the question reduces to the equality case in Lemma \ref{lem:Gpol}. The only inequality in its proof is $\sup^n\leq\sum^n$ over the fibres of $f$, and equality for the extremal density of a condenser $(C,\partial U)$ would force essential injectivity of $f$ on the support of that density. Whether the extremal density is positive almost everywhere, equivalently whether the capacity potential of the condenser, which is $\mathcal{A}$-harmonic for an operator with bounded measurable coefficients, has a null critical set, is known for $n=2$ and open for $n\geq 3$.
\item Is $(\B^n,d_G)$ itself Gromov hyperbolic for $n\geq 3$? For $G\equiv I$ this is Lind\'en's theorem \cite{Lin07}, and for $n=2$ it follows from his result by the remark at the end of Section 7. Theorem \ref{thm:coarse} answers the question for the chain metric only. A positive answer for general $G$ would presumably require a substitute for the special-function estimates on which Lind\'en's proof relies.
\item The examples of Theorem \ref{thm:examples} are all of positive-step type. Do non-injective uniformly quasiregular self-maps of $\B^n$, $n\geq 3$, of zero-step type exist? For $n=2$, $z\mapsto (z^2+1)/2$ is an example. Conformal traps cannot produce one, since after one step such a map coincides with a M\"obius map, whose orbits have constant hyperbolic step.
\item The boundary approach of \cite{FM20} requires knowing that a non-identity uniformly quasiregular map cannot fix a non-degenerate continuum meeting its Julia set. Martin \cite{Mar02} shows that the fixed-point set of a non-identity uniformly quasiregular map has homological dimension at most $n-2$, and asks whether a continuum of fixed points can meet the Julia set at all. This question was raised to the author by Aimo Hinkkanen and, to the author's knowledge, this remains open.
\end{enumerate}

\end{document}